\documentclass[12pt]{amsart}
\usepackage{graphicx}
\usepackage{amssymb}
\usepackage{amsmath}
\usepackage{amsthm}
\usepackage{amscd}
\usepackage{epsfig}

\newtheorem{thm}{Theorem}[section]
\newtheorem{cor}[thm]{Corollary}
\newtheorem{lem}[thm]{Lemma}
\newtheorem{prop}[thm]{Proposition}
\newtheorem{ques}[thm]{Question}

\newtheorem{ex}[thm]{Example}
\newtheorem{rem}[thm]{Remark}
\theoremstyle{definition}
\newtheorem{defn}[thm]{Definition}
\numberwithin{equation}{section}

\newcommand{\N}{\mathbb N}

\def\R {\Bbb R}
\def\Z {\Bbb Z}

\def \ep {\epsilon}
\def \ra {\rightarrow}
\def \HUL {{\bf{HUL}}}

\begin{document}

\title{Lowering topological entropy over subsets}

\date{February 16, 2008}
\date{April 7, 2008}
\date{October 18, 2008}

\author{Wen Huang, Xiangdong Ye and Guohua Zhang}

\address{Department of Mathematics, University of Science and Technology of
China, Hefei, Anhui, 230026, P.R. China}

\email{wenh@mail.ustc.edu.cn, yexd@ustc.edu.cn}

\address{School of Mathematical Sciences, Fudan University, Shanghai 200433, China}

\email{zhanggh@fudan.edu.cn}



\thanks{The authors are supported by a grant
from Ministry of Education (20050358053), NSFC and  973 Project
(2006CB805903). The first author is supported by FANEDD (Grant No
200520) and the third author is supported by NSFC (10801035).}

\begin{abstract}
Let $(X, T)$ be a topological dynamical system (TDS), and $h (T, K)$
the topological entropy of a subset $K$ of $X$. $(X, T)$ is {\it
lowerable} if for each $0\le h\le h_{} (T, X)$ there is a non-empty
compact subset with entropy $h$; is {\it hereditarily lowerable} if
each non-empty compact subset is lowerable; is {\it hereditarily
uniformly lowerable} if for each non-empty compact subset $K$ and
each $0\le h\le h (T, K)$ there is a non-empty compact subset
$K_h\subseteq K$ with $h (T, K_h)= h$ and $K_h$ has at most one
limit point.

It is shown that each TDS with finite entropy is lowerable, and that
a TDS $(X, T)$ is hereditarily uniformly lowerable if and only if it
is asymptotically $h$-expansive.
\end{abstract}

\maketitle

%\tableofcontents

\markboth{lowering topological entropy over subsets}{Wen Huang,
Xiangdong Ye and Guohua Zhang}

\section{Introduction}

Throughout the paper, by a {\it topological dynamical system} (TDS)
$(X, T)$ we mean a compact metric space $X$ and a homeomorphism $T:
X\rightarrow X$ (in fact our main results hold for continuous maps,
see Appendix). Let $(X, T)$ be a TDS. It is an interesting question,
considered in \cite{SW} firstly, whether for any given $0\le h\le
h_{\text{top}} (T, X)$, there is a factor $(Y, S)$ of $(X, T)$ with
entropy $h$. We remark that the answer to this question in the
measure-theoretical setup is well known, but in the topological
setting the answer is not completely obtained yet. In \cite{SW} Shub
and Weiss presented an example with infinite entropy such that each
its non-trivial factor has infinite entropy. Moreover, Lindenstrauss
\cite{L1} showed that the question has an affirmative answer when
$X$ is finite-dimensional; and for an extension of non-trivial
minimal $\Z$-actions the question has an affirmative answer if it
has zero mean topological dimension \cite{L2} which includes
finite-dimensional systems, systems with finite entropy and uniquely
ergodic systems. For the definition and properties of mean
topological dimension see \cite{LW} by Lindenstrauss and Weiss.

Let $(X, T)$ be a TDS and $K\subseteq X$. Denote by $h (T, K)$ the
topological entropy of $K$. In this paper we study a question
similar to the above one. Namely, we consider the question if for
each $0\le h\le h_{} (T, X)$ there is a non-empty compact subset of
$X$ with entropy $h$. We remark that the question was motivated by
\cite{SW, L1, LW, L2} and the well-known result in fractal geometry
\cite{Fal,M} which states that if $K$ is a non-empty Borel subset
contained in $\R^n$ then for each $0\le h\le dim_H(K)$ there is a
Borel subset $K_h$ of $K$ with $dim_H(K_h)=h$, where $dim_H(*)$ is
the Hausdorff dimension of a subset $*$ of $\R^n$.

In \cite{YZ} Ye and Zhang introduced and studied the notion of
entropy points, and showed that for each non-empty compact subset
$K$ there is a countable compact subset $K_1\subseteq K$ with $h (T,
K_1)= h (T, K)$. Moreover, the subset can be chosen such that the
limit points of the subset has at most one limit point (for details
see \cite[Remark 5.13]{YZ}). Inspired by this fact we have the
following notions.

\begin{defn}
Let $(X, T)$ be a TDS. We say that $(X, T)$ is
\begin{enumerate}
\item {\it lowerable} if for each $0\le h\le h_{} (T,
X)$ there is a non-empty compact subset of $X$ with entropy $h$;

\item {\it hereditarily
lowerable} if each non-empty compact subset is lowerable, that
is, for each non-empty compact subset $K\subseteq X$ and each
$0\le h\le h (T, K)$ there is a non-empty compact subset $K_h$
of $K$ with entropy $h$;

\item {\it hereditarily
uniformly lowerable} (\HUL\ for short) if for each non-empty
compact subset $K$ and each $0\le h\le h (T, K)$ there is a
non-empty compact subset $K_h\subseteq K$ such that $h (T, K_h)=
h$ and $K_h$ has at most one limit point.
\end{enumerate}
\end{defn}

So our question can be divided further into the following questions.

\begin{ques} \label{q 1}
Is any TDS lowerable?
\end{ques}

\begin{ques} \label{q 2}
Is any TDS hereditarily lowerable?
\end{ques}

\begin{ques} \label{q 3}
For which TDS it is \HUL?
\end{ques}

We remark that lowering  entropy for factors is not the same as
lowering entropy for subsets. For example, in \cite{L1}
Lindenstrauss showed that each non-trivial factor of $([0,1]^\Z,
\sigma)$ has infinite entropy, where $\sigma$ is the shift. But
since $(\{0,1, \ldots, k\}^\Z, \sigma)$ can be embedded as a
sub-system of $([0,1]^\Z, \sigma)$ for any $k\ge 1$, it is clear
that $([0,1]^\Z, \sigma)$ is lowerable in our sense.

In this paper, we show that each TDS with finite entropy is
lowerable (this is also true when we talk about the dimensional
entropy of a subset), and that a TDS is \HUL\ iff it is
asymptotically $h$-expansive. In particular, each \HUL\ TDS has
finite entropy. Moreover, a principal extension preserves the
lowerable, hereditarily lowerable and \HUL\ properties. It is not
hard to construct examples with infinite entropy which are
hereditarily lowerable. Thus, there are TDSs which are hereditarily
lowerable but not \HUL. In fact, an example with the same property
is explored at the end of the paper, which has finite entropy. The
questions remain open if there are lowerable but not hereditarily
lowerable examples, or there are TDSs with infinite entropy which
are not lowerable. We should remark that if $([0,1]^\Z, \sigma)$ is
hereditarily lowerable then each finite dimensional TDS without
periodic points is hereditarily lowerable (see \cite{L2}), and if it
is not then it is a lowerable TDS with infinite entropy which is not
hereditarily lowerable. We also remark that if there exists a TDS
which is not lowerable (such a TDS, if exists, must have infinite
entropy) then we can obtain a lowerable TDS with infinite entropy
which is not hereditarily lowerable by considering the union of it
and $([0,1]^\Z, \sigma)$. There are also many other interesting
questions related to the topic.

\medskip
The paper is organized as follows. In section 2 the definitions of
topological entropy and dimensional entropy of subsets are given,
and some basic properties are discussed. In the following section
two distribution principles are stated which will be used in section
4, where it is shown that each TDS with finite entropy is lowerable
by using the principles and a conditional version of
Shannon-McMillan-Breiman Theorem. The next three sections are
devoted to prove that a TDS is \HUL\ iff it is asymptotically
$h$-expansive, and the main ingredients of which are some techniques
developed in \cite{YZ, BD, DS, LW1}. An example with finite entropy
which is hereditarily lowerable but not \HUL\ is presented at the
end of paper.

\medskip

We thank D. Feng \cite{F} for asking the question: whether each
non-empty compact subset is lowerable? His question gave us the
first motivation of the research. We also thank the referees of the
paper for their careful reading and useful suggestions which greatly
improved the writing of the paper.

\section{Preliminary}

The discussions in this section and next section proceed for a {\it
general TDS} (GTDS), by a GTDS $(X, T)$ we mean a compact metric
space $X$ and a continuous mapping $T: X\rightarrow X$.

\medskip

Let $(X, T)$ be a GTDS, $K\subseteq X$ and $\mathcal{W}$ a family of
subsets of $X$. Set $\text{diam} (K)$ to be the diameter of $K$ and
put $||\mathcal{W}||= \sup \{\text{diam} (W): W\in \mathcal{W}\}$.
We shall write $K\succeq \mathcal{W}$ if $K\subseteq W$ for some
$W\in \mathcal{W}$ and else $K\nsucceq \mathcal{W}$. If
$\mathcal{W}_1$ is another family of subsets of $X$, $\mathcal{W}$
is said to be {\it finer} than $\mathcal{W}_1$ (we shall write
$\mathcal{W}\succeq \mathcal{W}_1$) when $W\succeq \mathcal{W}_1$
for each $W\in \mathcal{W}$. We shall say that a numerical function
{\it increases} (resp. {\it decreases}) with respect to (w.r.t.) a
set variable $K$ or a family variable $\mathcal{W}$ if the value
never decreases (resp. increases) when $K$ is replaced by a set
$K_1$ with $K_1\subseteq K$ or when $\mathcal{W}$ is replaced by a
family $\mathcal{W}_1$ with $\mathcal{W}_1\succeq \mathcal{W}$. By a
{\it cover} of $X$ we mean a finite family of Borel subsets with
union $X$ and a {\it partition} a cover whose elements are disjoint.
Denote by $\mathcal{C}_X$ (resp. $\mathcal{C}^o_X$, $\mathcal{P}_X$)
the set of covers (resp. open covers, partitions). Observe that if
$\mathcal{U}\in \mathcal{C}^o_X$ then $\mathcal{U}$ has a Lebesgue
number $\lambda> 0$ and so $\mathcal{W}\succeq \mathcal{U}$ when
$||\mathcal{W}||< \lambda$. If $\alpha\in \mathcal{P}_X$ and $x\in
X$ then let $\alpha (x)$ be the element of $\alpha$ containing $x$.
Given $\mathcal{U}_1, \mathcal{U}_2\in \mathcal{C}_X$, set
$\mathcal{U}_1\vee \mathcal{U}_2= \{U_1\cap U_2: U_1\in
\mathcal{U}_1, U_2\in \mathcal{U}_2\}$, obviously $\mathcal{U}_1\vee
\mathcal{U}_2\in \mathcal{C}_X$ and $\mathcal{U}_1\vee
\mathcal{U}_2\succeq \mathcal{U}_1$. $\mathcal{U}_1\succeq
\mathcal{U}_2$ need not imply that $\mathcal{U}_1\vee \mathcal{U}_2=
\mathcal{U}_1$, $\mathcal{U}_1\succeq \mathcal{U}_2$ iff
$\mathcal{U}_1$ is equivalent to $\mathcal{U}_1\vee \mathcal{U}_2$
in the sense that each refines the other. For each $\mathcal{U}\in
\mathcal{C}_X$ and any $m, n\in \mathbb{Z}_+$ with $m\le n$ we set
$\mathcal{U}_m^n= \bigvee_{i= m}^n T^{- i} \mathcal{U}$.

The following obvious fact will be used in several places and is
easy to check.

\begin{lem}
Let $\mathcal{V}\in \mathcal{C}^o_X$ and $\{\mathcal{U}_n: n\in
\N\}\subseteq \mathcal{C}_X$. If $||\mathcal{U}_n||\rightarrow 0$ as
$n\rightarrow +\infty$ then there exists $n_0\in \N$ such that
$\mathcal{U}_n\succeq \mathcal{V}$ for each $n\ge n_0$.
\end{lem}

\subsection{Topological entropy of subsets}\

Let $(X, T)$ be a GTDS, $K\subseteq X$ and $\mathcal{U}\in
\mathcal{C}_X$. Set $N (\mathcal{U}, K)$ to be the minimal
cardinality of sub-families $\mathcal{V}\subseteq \mathcal{U}$ with
$\cup \mathcal{V}\supseteq K$, where $\cup\mathcal{V}= \bigcup_{V\in
\mathcal{V}}V$. We write $N(\mathcal{U},\emptyset)=0$ by convention.
Obviously, $N (\mathcal{U}, T (K))= N (T^{- 1} \mathcal{U}, K)$. Let
\begin{equation*}
h_\mathcal{U} (T, K)= \limsup_{n\rightarrow +\infty} \frac{1}{n}
\log N (\mathcal{U}_0^{n- 1}, K).
\end{equation*}
Clearly $h_\mathcal{U} (T, K)$ increases w.r.t. $\mathcal{U}$.
Define the {\it topological entropy of $K$} by
$$h (T, K)= \sup_{\mathcal{U}\in \mathcal{C}^o_X} h_\mathcal{U} (T, K),$$
 and define the {\it
topological entropy of $(X, T)$} by $h_{\text{top}} (T, X)= h (T,
X)$.

Let $Z$ be a topological metric space and $f: Z\rightarrow [-
\infty, + \infty]$ a generalized real-valued function on $Z$. The
function $f$ is called {\it upper semi-continuous} (u.s.c. for
short) if $\{z\in Z: f (z)\ge r\}$ is a closed subset of $Z$ for
each $r\in \mathbb{R}$, equivalently,
\begin{equation*}
\limsup_{z'\rightarrow z} f (z')\le f (z)\ \text{for each}\ z\in Z.
\end{equation*}
Thus, the infimum of any family of u.s.c. functions is again a
u.s.c. one, both the sum and supremum of finitely many u.s.c.
functions are u.s.c. ones. In particular, the infimum of any family
of continuous functions is a u.s.c. function.

Let $(X, T)$ be a GTDS and $2^X$ its hyperspace, that is, $$2^X=
\{K: K\ \text{is a non-empty compact subset of}\ X\}.$$ We endow the
Hausdorff metric on $2^X$. Then $T$ induces a continuous mapping
$\widehat{T}$ on $2^X$ by $\widehat{T} (K)= T K$. The {\it entropy
hyper-function} $H: 2^X\rightarrow [0, h_{\text{top}} (T, X)]$ of
$(X, T)$ is defined by $H (K)= h (T, K)$ for $K\in 2^X$. Then we
have the follow results.

\begin{prop} \label{hyper-function-measurable}
Let $(X, T)$ be a GTDS and $\mathcal{U}\in \mathcal{C}^o_X$. Then
\begin{enumerate}

\item
$h_\mathcal{U} (T, K)= h_\mathcal{U} (T, T K)$ for any $K\subseteq
X$. Moreover, the entropy hyper-function $H$ is
$\widehat{T}$-invariant.

\item The function
$h_\mathcal{U} (T, \bullet)$ is Borel measurable on $2^X$.

\item The entropy hyper-function $H$ is Borel measurable.
\end{enumerate}
\end{prop}
\begin{proof}
(1) is clear. (2) follows from the following fact that for any
$\mathcal{V}\in \mathcal{C}_X^o$, $N(\mathcal{V}, \bullet): K\in
2^X\mapsto N(\mathcal{V},K)$ is a u.s.c function on $2^X$.  (3)
comes from (2).
\end{proof}

We may also obtain the topological entropy of subsets using Bowen's
separated and spanning sets (see \cite[P$_{168-174}$]{Wa}). Let $(X,
T)$ be a TDS with $d$ a metric on $X$. For each $n\in \mathbb{N}$ we
define a new metric $d_n$ on $X$ by
$$d_n (x, y)= \max_{0\le i\le n- 1}d (T^i x, T^i y).$$ Let
$\epsilon> 0$ and $K\subseteq X$. A subset $F$ of $X$ is said to
{\it $(n,\epsilon)$-span $K$ w.r.t. $T$} if for each $ x\in K$,
there is $y\in F$ with $d_n(x,y)\le\epsilon$; a subset $E$ of $K$ is
said to be {\it $(n,\epsilon)$-separated w.r.t. $T$} if $x,y\in E,
x\neq y$ implies $d_n(x,y)>\epsilon$. Let $r_n(d,T,\epsilon,K)$
denote the smallest cardinality of any $(n,\epsilon)$-spanning set
for $K$ w.r.t. $T$ and $s_n(d,T,\epsilon,K)$ denote the largest
cardinality of any $(n,\epsilon)$-separated subset of $K$ w.r.t.
$T$. We write $r_n(d,T,\epsilon,\emptyset)=s_n (d, T, \epsilon,
\emptyset)=0$ by convention. Put
$$r (d, T, \epsilon, K)= \limsup_{n\rightarrow +\infty} \frac{1}{n}
\log r_n (d, T, \epsilon, K)$$ and
$$ s (d, T, \epsilon, K)= \limsup_{n\rightarrow +\infty} \frac{1}{n}
\log s_n (d, T, \epsilon, K).$$ Then put $$h_* (d, T, K)=
\lim_{\epsilon\rightarrow 0+} r (d, T, \epsilon, K)\ \text{and}\ h^*
(d, T, K)= \lim_{\epsilon\rightarrow 0+} s (d, T, \epsilon, K).$$ It
is  well known that $h_* (d, T, K)= h^* (d, T, K)$ is independent of
the choice of a compatible metric $d$ on the space $X$. Now, if
$\mathcal{U}\in \mathcal{C}^o_X$ has a Lebesgue number $\delta> 0$
then, for any $\delta'\in (0, \frac{\delta}{2})$ and each
$\mathcal{V}\in \mathcal{C}^o_X$ with $||\mathcal{V}||\le \delta'$,
one has
$$N(\mathcal{U}_0^{n- 1}, K)\le r_n (d, T, \delta', K)\le s_n (d, T,
\delta', K)\le N(\mathcal{V}_0^{n- 1}, K)$$ for each $n\in
\mathbb{N}$. So if $\{\mathcal{U}_n\}_{n\in \mathbb{N}}\subseteq
\mathcal{C}^o_X$ satisfies $||\mathcal{U}_n||\rightarrow 0$ as
$n\rightarrow +\infty$ then
\begin{equation*}
h_* (d, T, K)=h^* (d, T, K)=\lim_{n\rightarrow +\infty}
h_{\mathcal{U}_n} (T, K)= h (T, K).
\end{equation*}
It is also obvious that $h (T, \overline{K})= h (T, K)$.

\subsection{Dimensional entropy of subsets}\

In the process of proving that each TDS with finite entropy is
lowerable, we shall use some concept named dimensional entropy of
subsets, which is another kind of topological entropy introduced and
studied in \cite{B3}. Let's see how to define it.

Let $(X, T)$ be a GTDS and $\mathcal{U}\in \mathcal{C}_X$. For
$K\subseteq X$ let
\[
n_{T, \mathcal{U}} (K)=\left\{
\begin{array}{ll}
0, &\mbox{if $K\nsucceq \mathcal{U}$};\\
+\infty, &\mbox{if $T^i K\succeq \mathcal{U}$ for all $i\in \mathbb{Z}_+$};\\
k, &\mbox{$k=\max\{j\in\N: T^i(K)\succeq \mathcal{U}\ \text{for
each}\ 0\le i\le j-1 \}$}.
\end{array}
\right.
\]
For $k\in \mathbb{N}$, we define
\begin{eqnarray*}
\mathfrak{C} (T,\mathcal{U},K, k)&=&
 \{ \mathcal{E}:\mathcal{E}\ \text{is a countable family of
subsets of}\ X   \\
&&\hskip1.5cm \text{ such that} K\subseteq \cup \mathcal{E} \text{
and } \mathcal{E}\succeq \mathcal{U}_0^{k-1} \}.
\end{eqnarray*}
Then for each $\lambda\in \mathbb{R}$ set
\begin{equation*}
m_{T, \mathcal{U}} (K, \lambda,k)=\inf_{\mathcal{E}\in
\mathfrak{C} (T, \mathcal{U},K, k)} m (T, \mathcal{U},
\mathcal{E}, \lambda),
\end{equation*}
where $m (T, \mathcal{U}, \mathcal{E}, \lambda)= \sum_{E\in
\mathcal{E}} e^{-\lambda n_{T, \mathcal{U}} (E)}$ and we write $m
(T,\mathcal{U}, \emptyset,\lambda)=0$ by convention. As
$m_{T,\mathcal{U}}(K,\lambda,k)$ is decreasing w.r.t.
$k$, we can define
$$m_{T,\mathcal{U}}(K,\lambda)=\lim_{k\rightarrow +\infty}
m_{T, \mathcal{U}} (K, \lambda,k).$$ Notice that
$m_{T,\mathcal{U}}(K,\lambda)\le m_{T,\mathcal{U}}(K,\lambda')$ for
$\lambda\ge \lambda'$ and $m_{T, \mathcal{U}} (K, \lambda)\notin
\{0, +\infty\}$ for at most one $\lambda$ \cite{B3}. We define {\it
the dimensional entropy of $K$ relative to $\mathcal{U}$} by
\begin{equation*}
h^B_\mathcal{U} (T, K)= \inf \{\lambda\in \mathbb{R}: m_{T, \mathcal{U}}
(K, \lambda)= 0\}= \sup \{\lambda \in \mathbb{R}: m_{T, \mathcal{U}} (K,
\lambda)= +\infty\}.
\end{equation*}
The {\it dimensional entropy of $K$} is defined by
$$h^B (T, K)= \sup_{\mathcal{U}\in \mathcal{C}^o_X} h_\mathcal{U}^B (T, K).$$
Note that $h_{\mathcal{U}}^B (T, K)$ increases w.r.t.
$\mathcal{U}\in \mathcal{C}_X$, thus if $\{\mathcal{U}_n\}_{n\in
\mathbb{N}}\subseteq \mathcal{C}^o_X$ satisfies $\lim_{n\rightarrow
+\infty} ||\mathcal{U}_n||=0$ then $\lim_{n\rightarrow +\infty}
h^B_{\mathcal{U}_n} (T, K)=h^B (T, K)$.

The following results are elementary (see for example
\cite[Propositions 1 and 2]{B3}).

\begin{prop} \label{06.02.28}
Let $(X, T)$ be a GTDS, $K_1, K_2, \cdots, K\subseteq X$ and
$\mathcal{U}\in \mathcal{C}_X$. Then
\begin{enumerate}

\item $h_\mathcal{U} (T, X)= h^B_\mathcal{U} (T, X)$ if $\mathcal{U}
\in \mathcal{C}_X^o$, so $h (T, X)= h^B (T, X)$.

\item $h^B_\mathcal{U} (T, \bigcup_{n\in \mathbb{N}}
K_n)= \sup_{n\in \mathbb{N}} h^B_\mathcal{U} (T, K_n)$, so
$$h^B (T,
\bigcup_{n\in \mathbb{N}} K_n)= \sup_{n\in \mathbb{N}} h^B (T,
K_n).$$

\item For each $m\in \mathbb{N}$ and $i\ge 0$,
$h^B_{T^{-i}\mathcal{U}} (T^m, K)\ge h^B_\mathcal{U} (T^m, T^iK)$,
so $h^B (T^m, K)\ge h^B (T^m, T^i K)$.

\item For each $m\in
\mathbb{N}$, $h^B_{\mathcal{U}_0^{m- 1}} (T^m, K)= m
h^B_\mathcal{U} (T, K)$, so $h^B (T^m, K)= m h^B (T, K)$.

\end{enumerate}
\end{prop}
\begin{proof}
(1) is \cite[Proposition 1]{B3}. (2) is obvious.

(3) Let $m\in \mathbb{N}$ and $i\ge 0$. Assume $k\in \mathbb{N}$ and
$\lambda>0$. If $\mathcal{E}\in \mathfrak{C}
(T^m,T^{-i}\mathcal{U},K, k)$ then $n_{T^m,\mathcal{U}}(T^i
E)=n_{T^m,T^{-i}\mathcal{U}}(E)\ge k$ for each $E\in \mathcal{E}$
and so
$$T^i(\mathcal{E})\doteq \{ T^iE:E\in \mathcal{E}\}\in \mathfrak{C}
(T^m,\mathcal{U},T^iK, k),$$ thus
\begin{eqnarray*}
m_{T^m, \mathcal{U}}(T^iK, \lambda,k)&\le & m (T^m,
\mathcal{U}, T^i (\mathcal{E}), \lambda)= \sum_{E\in \mathcal{E}}
e^{-\lambda n_{T^m, \mathcal{U}}
(T^iE)}\\
&=& \sum_{E\in \mathcal{E}} e^{-\lambda n_{T^m,
T^{-i}\mathcal{U}}(E)}= m (T^m, T^{-i}\mathcal{U}, \mathcal{E},
\lambda),
\end{eqnarray*}
which implies $m_{T^m, \mathcal{U}}(T^iK, \lambda, k)\le
m_{T^m, T^{-i}\mathcal{U}} (K, \lambda,k)$ as $\mathcal{E}$
is arbitrary. Letting $k\rightarrow +\infty$ we get $m_{T^m,
\mathcal{U}}(T^iK, \lambda)\le m_{T^m, T^{-i}\mathcal{U}}(K,
\lambda)$, hence $h^B_{\mathcal{U}} (T^m, T^iK)\le
h^B_{T^{-i}\mathcal{U}}(T^m, K)$, as $\lambda>0$ is arbitrary.

(4) Let $m\in \mathbb{N}$ and $n\in \N$, $\lambda> 0$. If
$\mathcal{E}\in \mathfrak{C} (T,\mathcal{U},K,m n)$ then
$$n_{T^m, \mathcal{U}_0^{m- 1}} (E)= \left[\frac{n_{T, \mathcal{U}}
(E)}{m}\right]\ge \max \left\{n, \frac{n_{T, \mathcal{U}} (E)}{m}-
\frac{m- 1}{m}\right\}$$ for each $E\in \mathcal{E}$, where $[a]$
denotes the integral part of a real number $a$, so
$$\inf_{E\in \mathcal{E}} n_{T^m, \mathcal{U}_0^{m- 1}} (E)\ge
n,$$ thus $\mathcal{E}\in \mathfrak{C} (T^m,\mathcal{U}_0^{m- 1},K,
n)$ and
\begin{eqnarray*}
m_{T^m, \mathcal{U}_0^{m- 1}} (K, \lambda,
n)&\le & m (T^m, \mathcal{U}_0^{m- 1}, \mathcal{E},
\lambda)= \sum_{E\in \mathcal{E}} (e^\lambda)^{- n_{T^m,
\mathcal{U}_0^{m- 1}}
(E)}\\
&\le & \sum_{E\in \mathcal{E}} (e^\lambda)^{\frac{m-
1}{m}- \frac{n_{T, \mathcal{U}} (E)}{m}}= e^{\frac{(m- 1)
\lambda}{m}}\cdot m (T, \mathcal{U}, \mathcal{E},
\frac{\lambda}{m}),
\end{eqnarray*}
which implies $m_{T^m, \mathcal{U}_0^{m- 1}}(K, \lambda,
n)\le e^{\frac{(m- 1) \lambda}{m}} m_{T, \mathcal{U}} (K,
\frac{\lambda}{m}, m n)$ as $\mathcal{E}$ is arbitrary. We
get
$$m_{T^m, \mathcal{U}_0^{m- 1}}(K, \lambda)\le e^{\frac{(m- 1)
\lambda}{m}}\cdot m_{T, \mathcal{U}} (K,
\frac{\lambda}{m})$$ by letting $n\rightarrow +\infty$, hence
$h^B_{\mathcal{U}_0^{m- 1}} (T^m, K)\le m h^B_\mathcal{U} (T, K)$,
as $\lambda> 0$ is arbitrary.

Following similar discussions we obtain $m_{T, \mathcal{U}} (K,
\lambda)\le m_{T^m, \mathcal{U}_0^{m- 1}} (K, m\lambda)$ for each
$\lambda> 0$, then $h^B_\mathcal{U} (T, K)$ $\le \frac{1}{m}
h^B_{\mathcal{U}_0^{m- 1}} (T^m, K)$. That is,
$h^B_{\mathcal{U}_0^{m- 1}} (T^m, K)= m h^B_\mathcal{U} (T, K)$.
\end{proof}
 By Proposition \ref{06.02.28} (2), $h^B (T, E)$ increases w.r.t.
 $E\subseteq X$. At the same time,
if $E\subseteq X$ is a non-empty countable set then $h^B(T,E)=0$.
Finally, it is worth mentioning that a).
$h_{\mathcal{U}}^B(T,\emptyset)=h_{\mathcal{U}}(T,\emptyset)=-\infty$
for any  $\mathcal{U}\in \mathcal{C}_X$, and so
$h^B(T,\emptyset)=h(T,\emptyset)=-\infty$; b). when $\emptyset
\neq K\subseteq X$, one has $h_{\mathcal{U}}(T,K)\ge
h_{\mathcal{U}}^B(T,K)\ge 0$ for any $\mathcal{U}\in
\mathcal{C}_X$, and so $h(T,K)\ge h^B(T,K)\ge 0$.

\section{Distribution principles}

In this section we shall present two important distribution
principles which link Question \ref{q 1} with ergodic theory and
play a key role in the next section. We remark that the distribution
principles were essentially contained in \cite{P}.

\medskip

The first result is an  obvious link between two definitions of
entropy.

\begin{lem}[{\bf Bridge Lemma}] \label{bridge}
Let $(X, T)$ be a GTDS, $\mathcal{U}\in \mathcal{C}_X$ and
$K\subseteq X$. Then
\begin{equation*}
h_\mathcal{U}^B (T, K)\le \liminf_{n\rightarrow +\infty} \frac{1}{n}
\log N (\mathcal{U}_0^{n- 1}, K)\le h_\mathcal{U} (T, K).
\end{equation*}
\end{lem}
\begin{proof} When $K=\emptyset$, this is clear. Now we assume
$K\neq \emptyset$. For each $n\in \mathbb{N}$ let $\mathcal{T}_n=
\{A_1, \cdots, A_{N (\mathcal{U}_0^{n- 1}, K)}\}\subseteq
\mathcal{U}_0^{n- 1}$ such that $\cup \mathcal{T}_n\supseteq K$. As
$n_{T, \mathcal{U}} (A)\ge n$ for each $A\in \mathcal{T}_n$, for
each $\lambda\ge 0$ one has
\begin{equation*}
m_{T, \mathcal{U}} \left( K, \lambda, n \right)\le
\sum_{A\in \mathcal{T}_n} \left(e^\lambda\right)^{- n_{T,
\mathcal{U}} (A)}\le \sum_{A\in \mathcal{T}_n} (e^\lambda)^{- n}= N
(\mathcal{U}_0^{n- 1}, K) e^{- \lambda n},
\end{equation*}
then
\begin{equation*}
m_{T, \mathcal{U}} (K, \lambda)\le \liminf_{n\rightarrow +\infty} N
(\mathcal{U}_0^{n- 1}, K) e^{- \lambda n}= \liminf_{n\rightarrow
+\infty} e^{- n (\lambda- \frac{1}{n} \log N (\mathcal{U}_0^{n- 1},
K))}.
\end{equation*}
So, if $\lambda> \liminf\limits_{n\rightarrow +\infty} \frac{1}{n}
\log N (\mathcal{U}_0^{n- 1}, K)$ then $m_{T, \mathcal{U}} (K,
\lambda)= 0$, which ends the proof.
\end{proof}

Let $(X, T)$ be a GTDS, $\mathcal{U}\in \mathcal{C}_X$, $K\subseteq
X$ and $n\in \mathbb{N}$. Set $\mathfrak{M} (T,\mathcal{U}, K,n)$
to be the collection of all countable families $\mathcal{T}$ of
subsets of $X$ with
\begin{equation*}
\cup \mathcal{T}\supseteq K\ \text{and for each}\ A\in \mathcal{T},
A\cap K\neq \emptyset, n_{T, \mathcal{U}} (A)\ge n\ \text{and}\ A\in
\mathcal{U}_0^{n_{T, \mathcal{U}} (A)- 1}.
\end{equation*}
Then for each $\lambda\in \mathbb{R}$ set
\begin{equation*}
f_{T, \mathcal{U}} (K, \lambda)= \lim_{n\rightarrow +\infty}\, \inf_{\mathcal{T}\in
\mathfrak{M} (T,\mathcal{U}, K,n)} m(T, \mathcal{U}, \mathcal{T}, \lambda).
\end{equation*}
It's not hard to check that $f_{T, \mathcal{U}} (K, \lambda)= m_{T,
\mathcal{U}} (K, \lambda)$ for $\lambda\in \mathbb{R}$. In fact, for
$\mathcal{E}\in \mathfrak{C} (T,\mathcal{U},K, n)$,
note that for each $E\in \mathcal{E}$ there exists $\widetilde{E}\in
\mathcal{U}_0^{n_{T, \mathcal{U}} (E)- 1}$ with $E\subseteq
\widetilde{E}$ and so $n_{T, \mathcal{U}} (E)= n_{T, \mathcal{U}}
(\widetilde{E})$. Then let $\mathcal{T}\doteq \{\widetilde{E}: E\in
\mathcal{E} \text{ with }E\cap K\neq \emptyset\}$. Then
$\mathcal{T}\in \mathfrak{M} (T,\mathcal{U}, K,n)$. Particularly, when $E\cap K\neq \emptyset$ for each
$E\in \mathcal{E}$, one has $m (T,
\mathcal{U}, \mathcal{T}, \lambda)= m (T, \mathcal{U}, \mathcal{E},
\lambda)$. This implies $f_{T, \mathcal{U}} (K, \lambda)= m_{T,
\mathcal{U}} (K, \lambda)$.

\medskip
For a GTDS $(X,T)$, denote by $\mathcal{M}(X)$ the set of all Borel
probability measures on $X$. The following two principles will be
proved to be very useful.

\begin{lem}[{\bf Non-Uniform Mass Distribution Principle}] \label{non-uniform}
Let $(X, T)$ be a GTDS, $d> 0, M\in \mathbb{N}$, $Z\subseteq X$,
$\alpha\in \mathcal{P}_X, \mathcal{U}\in \mathcal{C}_X$ and
$\theta\in \mathcal{M} (X)$. Assume that each element of
$\mathcal{U}$ has a non-empty intersection with at most $M$ elements
of $\alpha$. If there exists $Z_\theta\subseteq Z$ such that
$Z_\theta$ has positive outer $\theta$-measure (i.e. $\theta^*
(Z_\theta)> 0$) and
\begin{equation*}
\forall x\in Z_\theta, \exists c (x)> 0\ \text{such that}\ \forall
n\in \mathbb{N}, \theta (\alpha_0^{n- 1} (x))\le c (x) e^{- n d}.
\end{equation*}
Then $h^B_\mathcal{U} (T, Z)\ge d- \log M$. In particular,
$h^B_\alpha (T, Z)\ge d$.
\end{lem}
\begin{proof}
It makes no difference to assume $d- \log M>0$. For each $k\in
\mathbb{N}$ set $Z_\theta^k= \{x\in Z_\theta: c (x)\le k\}$. Then
for some $N\in \mathbb{N}$, $\theta^* (Z_\theta^N)> 0$, as
$Z_\theta^1\subseteq Z_\theta^2\subseteq \cdots$, $Z_\theta=
\bigcup_{k\in \mathbb{N}} Z_\theta^k$ and $\theta^* (Z_\theta)> 0$.

Let $n\in \mathbb{N}$ and $\mathcal{T}\in \mathfrak{M} (T,
\mathcal{U}, Z,n)$. If $A\in \mathcal{T}$ satisfies $A\cap
Z_\theta^N\neq \emptyset$, then for each $s\in \mathbb{N}$ and $B\in
\alpha_0^{\min \{n_{T, \mathcal{U}} (A), s\}- 1}$ with $B\cap (A\cap
Z_\theta^N)\neq \emptyset$ select $x_B\in B\cap (A\cap Z_\theta^N)$,
so
\begin{equation*}
\theta (B)= \theta (\alpha_0^{\min \{n_{T, \mathcal{U}} (A), s\}- 1}
(x_B))\le c (x_B) e^{- \min \{n_{T, \mathcal{U}} (A), s\} d}\le N
e^{- \min \{n_{T, \mathcal{U}} (A), s\} d}.
\end{equation*}
Since there are at most $M^{\min \{n_{T, \mathcal{U}} (A), s\}}$
elements of $\alpha_0^{\min \{n_{T, \mathcal{U}} (A), s\}- 1}$ which
have non-empty intersection with $A\cap Z_\theta^N$, we have
$$\theta^* (A\cap Z_\theta^N)\le M^{\min \{n_{T, \mathcal{U}}
(A), s\}} N e^{- \min \{n_{T, \mathcal{U}} (A), s\} d}= N e^{- \min
\{n_{T, \mathcal{U}} (A), s\} (d- \log M)}.$$ Letting $s\rightarrow
+\infty$ we obtain that $\theta^* (A\cap Z_\theta^N)\le N e^{- n_{T,
\mathcal{U}} (A) (d- \log M)}$ for any $A\in \mathcal{T}$ satisfying
$A\cap Z_\theta^N\neq \emptyset$. Moreover,
\begin{eqnarray*}
\sum_{A\in \mathcal{T}} e^{- n_{T, \mathcal{U}} (A) (d- \log M)}&\ge
& \sum_{A\in \mathcal{T}, A\cap Z_\theta^N\neq \emptyset} e^{-
n_{T, \mathcal{U}} (A) (d- \log M)} \\
&\ge & \sum_{A\in \mathcal{T}, A\cap Z_\theta^N\neq \emptyset}
\frac{\theta^* (A\cap Z_\theta^N)}{N}\ge \frac{1}{N} \theta^*
(Z_\theta^N)> 0.
\end{eqnarray*}
Since $n$ and $\mathcal{T}$ are arbitrary, $f_{T, \mathcal{U}} (Z,
d- \log M)\ge \frac{1}{N} \theta^* (Z_\theta^N)> 0$. So
$h^B_\mathcal{U} (T, Z)\ge d- \log M$.
\end{proof}

\begin{lem}[{\bf Uniform Mass Distribution Principle}] \label{uniform}
Let $(X, T)$ be a GTDS, $c> 0$, $d> 0$, $Z\subseteq X$ and
$\alpha\in \mathcal{P}_X$. If there exists $\theta\in \mathcal{M}
(X)$ such that for each $x\in Z$ and $n\in \mathbb{N}$, $\theta
(\alpha_0^{n- 1} (x))\ge c e^{- n d}$. Then $h_\alpha (T, Z) \le d$.
\end{lem}
\begin{proof} If $Z=\emptyset$ then $h_\alpha (T, Z)=-\infty \le d$. In
the following we assume $Z\neq \emptyset$. For $n\in \mathbb{N}$,
let $\mathcal{T}_n= \{A_1, \cdots, A_k\}$ be the collection of all
elements of $\alpha_0^{n- 1}$ which have non-empty intersection with
$Z$, where $k= N (\alpha_0^{n- 1}, Z)$. Take $x_i\in A_i\cap Z$,
then $\theta (A_i)= \theta (\alpha_0^{n- 1} (x_i))\ge c e^{- n d}$
for $i\in \{1,\cdots,k\}$. Therefore $1\ge \sum_{i= 1}^k \theta
(A_i)\ge k c e^{- n d}$, that is, $N (\alpha_0^{n- 1}, Z)= k\le
\frac{e^{n d}}{c}$. Finally letting $n\rightarrow +\infty$ we know
$h_\alpha (T, Z) \le d$.
\end{proof}

\section{Each TDS with finite entropy is lowerable}

In this section we shall give an affirmative answer to Question
\ref{q 1} for a TDS with finite entropy. In fact, we can obtain more
about it. Precisely, if $(X, T)$ is a TDS with finite entropy, then
for each $0\le h\le h_{\text{top}} (T, X)$ there exists a non-empty
compact subset $K_h\subseteq X$ such that $h^B (T, K_h)= h (T, K_h)=
h$ (for details see Theorem \ref{aim of paper}), particularly, $(X,
T)$ is lowerable.

\medskip

Let $(X, T)$ be a GTDS. Denote by $\mathcal{M}(X, T)$ and
$\mathcal{M}^{e}(X, T)$ the set of all $T$-invariant Borel
probability measures and ergodic $T$-invariant Borel probability
measures on $X$, respectively. Then $\mathcal{M}(X)$ and
$\mathcal{M}(X,T)$ are both convex, compact metric spaces when
endowed with the weak$^*$-topology. Denote by $\mathcal{B}_X$ the
set of all Borel subsets of $X$.

For any given $\alpha \in \mathcal{P}_X$, $\mu\in {\mathcal M}(X)$
and any sub-$\sigma$-algebra $\mathcal{C}\subseteq \mathcal
{B}_\mu$, where $\mathcal{B}_\mu$ is the completion of
$\mathcal{B}_X$ under $\mu$, the {\it conditional informational
function of $\alpha$ relevant to $\mathcal{C}$} is defined by
$$I_\mu(\alpha|\mathcal{C})(x)=\sum_{A\in \alpha}-1_A(x) \log \mathbb{E}_\mu (1_A|\mathcal{C})(x),$$
 where $\mathbb{E}_\mu (1_A|\mathcal{C})$ is the conditional
expectation of $1_A$ w.r.t. $\mathcal{C}$. Let
$$\ H_{\mu}(\alpha|\mathcal{C})=\int_X I_\mu(\alpha|\mathcal{C})(x) \, d\mu(x)=\sum_{A\in \alpha} \int_X
- \mathbb{E}_\mu (1_A|\mathcal{C}) \log
\mathbb{E}_\mu(1_A|\mathcal{C}) d \mu.$$ A standard fact states that
$H_\mu(\alpha| \mathcal{C})$ increases w.r.t. $\alpha$ and decreases
w.r.t. $\mathcal{C}$. Now set
\begin{align*}
H_\mu(\mathcal{U}|\mathcal{C})=\inf_{\beta\in \mathcal{P}_X,
\beta\succeq \mathcal{U}} H_\mu(\beta|\mathcal{C})
\end{align*}
for $\mathcal{U}\in \mathcal{C}_X$. Clearly,
$H_\mu(\mathcal{U}|\mathcal{C})$ increases w.r.t. $\mathcal{U}$ and
decreases w.r.t. $\mathcal{C}$.

 When $\mu\in \mathcal{M}(X,T)$ and $T^{-1}\mathcal{C}\subseteq \mathcal{C}$ in
the sense of $\mu$, it is not hard to see that
$H_\mu(\mathcal{U}_0^{n-1}|\mathcal{C})$ is a non-negative and
sub-additive sequence for a given $\mathcal{U}\in \mathcal{C}_X$,
so we can define
$$h_\mu(T,\mathcal{U}|\mathcal{C})=\lim_{n\rightarrow +\infty} \frac{1}{n}
H_\mu(\mathcal{U}_0^{n-1}|\mathcal{C})=\inf_{n\ge 1} \frac{1}{n}
H_\mu(\mathcal{U}_0^{n-1}|\mathcal{C}).$$ Clearly, $h_\mu(T,
\mathcal{U}|\mathcal{C})$ also increases w.r.t. $\mathcal{U}$ and
decreases w.r.t. $\mathcal{C}$.  The {\it relative
measure-theoretical $\mu$-entropy of $(X, T)$ relevant to
$\mathcal{C}$} is defined by
$$h_\mu (T, X| \mathcal{C})= \sup_{\alpha\in \mathcal{P}_X}
h_\mu (T, \alpha| \mathcal{C}).$$ Following a similar discussion of
\cite[Lemma 2.3 (1)]{HYZ}, one has $$h_\mu (T, X| \mathcal{C})=
\sup_{\mathcal{U}\in \mathcal{C}_X^o} h_\mu (T, \mathcal{U}|
\mathcal{C}).$$ So, if $\{\mathcal{U}_n\}_{n\in \mathbb{N}}\subseteq
\mathcal{C}_X^o$ satisfies $\lim_{n\rightarrow +\infty}
||\mathcal{U}_n||=0$ then $\lim_{n\rightarrow +\infty}h_\mu (T,
\mathcal{U}_n| \mathcal{C})$ $=h_\mu (T, X| \mathcal{C})$, as
$h_\mu(T, \mathcal{U}|\mathcal{C})$ increases w.r.t. $\mathcal{U}$.
It is not hard to see that
\begin{equation*} \label{fact3-06.03.04}
h_\mu (T, \mathcal{U}| \mathcal{C})= \frac{1}{n} h_\mu (T^n,
\mathcal{U}_0^{n- 1}| \mathcal{C})\text{ for each}\ n\in
\mathbb{N} \text{ and } \mathcal{U}\in \mathcal{C}_X.
\end{equation*}
If $\mathcal{C}=\{ \emptyset, X\}$ $(\text{mod}\ \mu)$, for
simplicity we shall write $H_{\mu}(\mathcal{U}|\mathcal{C})$,
 $h_\mu(T,\mathcal{U}|\mathcal{C})$ and $h_\mu(T,X|\mathcal{C})$ by
 $H_{\mu}(\mathcal{U})$, $h_\mu(T,\mathcal{U})$ and $h_\mu(T,X)$, respectively.

The following result is a conditional version of
Shannon-McMillan-Breiman Theorem. Its proof is completely similar to
the proof of Shannon-McMillan-Breiman Theorem (see e.g.
\cite[Theorem 4.2]{Bo}, \cite[Theorem II.1.5]{K}, \cite{G}).

\begin{thm} \label{SMB-conditional version} Let $(X, T)$ be a TDS,
$\mu\in \mathcal{M}(X, T)$, $\alpha\in \mathcal{P}_X$ and
$\mathcal{C}\subseteq \mathcal{B}_\mu$ a $T$-invariant
sub-$\sigma$-algebra (i.e. $T^{-1}\mathcal{C}=\mathcal{C}$ in the
sense of $\mu$). Then there exists a $T$-invariant function $f\in
L^1(\mu)$ such that $\int_X f (x) d\mu (x)= h_\mu (T, \alpha|
\mathcal{C})$ and
\begin{equation*}
\lim_{n\rightarrow +\infty} \frac{I_\mu (\alpha_0^{n- 1}|
\mathcal{C}) (x)}{n}= f (x)\ \mbox{for $\mu$-a.e. }x\in X\ \text{and
in}\ L^1 (\mu).
\end{equation*}
Moreover, if $\mu$ is ergodic then $f (x)= h_\mu (T, \alpha|
\mathcal{C})$ for $\mu$-a.e. $x\in X$.
\end{thm}

Let $(X,T)$ be a TDS, $\mu\in \mathcal{M}(X,T)$ and
$\mathcal{B}_\mu$ the completion of $\mathcal{B}_X$ under $\mu$.
Then $(X,\mathcal{B}_\mu,\mu,T)$ is a Lebesgue system. If $\{
\alpha_i\}_{i\in I}$ is a countable family of finite partitions of
$X$, the partition $\alpha=\bigvee_{i\in I}\alpha_i$ is called a
{\it measurable partition}. The sets $A\in \mathcal{B}_\mu$, which
are unions of atoms of $\alpha$, form a sub-$\sigma$-algebra of
$\mathcal{B}_\mu$ denoted by $\widehat{\alpha}$ or $\alpha$ if there
is no ambiguity. Every sub-$\sigma$-algebra of $\mathcal{B}_\mu$
coincides with a sub-$\sigma$-algebra constructed in this way (mod
$\mu$). Given a measurable partition $\alpha$, put
$\alpha^-=\bigvee_{n=1}^{+ \infty} T^{-n}\alpha$ and
$\alpha^T=\bigvee_{n=-\infty}^{+\infty} T^{-n}\alpha$. Define in the
same way $\mathcal{C}^-$ and $\mathcal{C}^T$ if $\mathcal{C}$ is a
sub-$\sigma$-algebra of $\mathcal{B}_\mu$. Clearly, for a measurable
partition $\alpha$, $\widehat{\alpha^-}=(\widehat{\alpha})^-$ (mod
$\mu$) and $\widehat{\alpha^T}=(\widehat{\alpha})^T$ (mod $\mu$).

Let $\mathcal{C}$ be a sub-$\sigma$-algebra of $\mathcal{B}_\mu$ and
$\alpha$ the measurable partition of $X$ with
$\widehat{\alpha}=\mathcal{C}$ (mod $\mu$). $\mu$ can be
disintegrated over $\mathcal{C}$ as $\mu=\int_X \mu_x d \mu(x)$
where $\mu_x\in \mathcal{M}(X)$ and $\mu_x(\alpha(x))=1$ for
$\mu$-a.e. $x\in X$. The disintegration is characterized by
\eqref{meas1} and \eqref{meas3} below: for every $f\in
L^1(X,\mathcal{B}_X,\mu)$,
\begin{eqnarray} \label{meas1}
&&f \in L^1(X,\mathcal{B}_X,\mu_x) \text{ for $\mu$-a.e. } x\in X\
\text{and}
\\&& \text{the function }x \mapsto \int_X  f(y) d\mu_x (y)\text{ is in
}L^1(X,\mathcal{C},\mu); \nonumber
\end{eqnarray}
\begin{equation}
\label{meas3} \mathbb{E}_{\mu}(f|\mathcal{C})(x)=\int_X f\,d\mu_{x}
\ \text{for $\mu$-a.e. } x\in X.
\end{equation}

Then, for any $f \in L^1(X,\mathcal{B}_X,\mu)$, the following
holds
\begin{equation*}
\int_X \left(\int_X f\,d\mu_x \right)\, d\mu(x)=\int_X f \,d\mu.
\end{equation*}
Define for $\mu$-a.e. $x\in X$ the set $\Gamma_x=\{y\in X :
\mu_x=\mu_y \}$. Then $\mu_x(\Gamma_x)=1$ for $\mu$-a.e. $x\in X$.
Hence given any $f \in L^1(X,\mathcal{B}_X,\mu)$, for $\mu$-a.e.
$x\in X$, one has
\begin{align}\label{meas4}
\mathbb{E}_{\mu}(f|\mathcal{C})(y)=\int_X f\,d\mu_{y}=\int_X
f\,d\mu_{x}=\mathbb{E}_{\mu}(f|\mathcal{C})(x)
\end{align}
for $\mu_x$-a.e. $y\in X$. Particularly, if $f$ is
$\mathcal{C}$-measurable, then for $\mu$-a.e. $x\in X$, one has
\begin{align}\label{meas5}
f(y)=f(x) \text{ for $\mu_x$-a.e. $y\in X$}.
\end{align}

\medskip

\begin{prop} \label{more estimate-not measurable}
Let $(X, T)$ be a TDS, $\mu\in \mathcal{M}(X, T)$ and
$\mathcal{C}\subseteq \mathcal{B}_\mu$ a $T$-invariant
sub-$\sigma$-algebra. If $\mu= \int_X \mu_x d \mu (x)$ is the
disintegration of $\mu$ over $\mathcal{C}$, then
\begin{enumerate}

\item Let $\mathcal{U}\in
\mathcal{C}_X$ and $\alpha\in \mathcal{P}_X$ such that each
element of $\mathcal{U}$ has a non-empty intersection with at most
$M$ elements of $\alpha$ ($M\in \mathbb{N}$). If
$f_{T,\alpha}^{\mathcal{C}}(x)$ is the function obtained in
Theorem \ref{SMB-conditional version} for $T$, $\alpha$ and
$\mathcal{C}$, then for $\mu$-a.e. $x\in X$,
\begin{align*}
h^B_\mathcal{U} (T, Z_x)\ge f_{T,\alpha}^{\mathcal{C}} (x)- \log M
\end{align*}
for any $Z_x\in \mathcal{B}_X$ with $\mu_x(Z_x)>0$. Particularly,
if $\mu$ is ergodic, then for $\mu$-a.e. $x\in X$,
\begin{align*}
h^B_\mathcal{U} (T, Z_x)\ge h_\mu(T,\alpha|\mathcal{C})-\log M
\end{align*}
for any $Z_x\in \mathcal{B}_X$ with $\mu_x(Z_x)>0$.

\item If $\mu$ is ergodic then, for $\mu$-a.e. $x\in X$, when $Z_x\in
\mathcal{B}_X$ with $\mu_x (Z_x)> 0$ one has $$h^B (T, Z_x)\ge h_\mu
(T, X| \mathcal{C}).$$
\end{enumerate}
\end{prop}
\begin{proof}
(1) Note that $\lim_{n\rightarrow +\infty}\frac{I_\mu (\alpha_0^{n-
1}| \mathcal{C}) (x)}{n}=f_{T,\alpha}^{\mathcal{C}}(x)$ for
$\mu$-a.e. $x\in X$ and $f_{T,\alpha}^{\mathcal{C}}$ is
$\mathcal{C}$-measurable, using \eqref{meas4} for all $1_B$, $B\in
\alpha_0^{n- 1}$ and \eqref{meas5} for $f_{T,\alpha}^{\mathcal{C}}$,
there exists $X_\infty\in \mathcal{B}_X$ with $\mu (X_\infty)= 1$
such that for each $x\in X_\infty$, one can find $W_x\in
\mathcal{B}_X$ with $\mu_x(W_x)=1$ and if $y\in W_x$ then
\begin{enumerate}

\item[(a).]  $\lim_{n\rightarrow +\infty} \frac{I_\mu (\alpha_0^{n- 1}|
\mathcal{C})
(y)}{n}=f_{T,\alpha}^{\mathcal{C}}(y)=f_{T,\alpha}^{\mathcal{C}}(x)$.

\item[(b).] $\mathbb{E}_\mu (1_B| \mathcal{C}) (y)=\mathbb{E}_\mu (1_B| \mathcal{C})
(x)=\mu_x (B)$ for any $B\in \alpha_0^{n- 1}$ and each $n\in
\mathbb{N}$.
\end{enumerate}
Moreover, for any $y\in W_x$, where $x\in X_\infty$, one has
\begin{align}\label{ee-kkk-gs}
\lim_{n\rightarrow +\infty} \frac{-\log
\mu_x(\alpha_0^{n-1}(y))}{n}&=\lim_{n\rightarrow +\infty}
\frac{-\log \mathbb{E}_\mu (1_{\alpha_0^{n-1}(y)}| \mathcal{C})
(y)}{n}\\
&=\lim_{n\rightarrow +\infty} \frac{I_\mu (\alpha_0^{n- 1}|
\mathcal{C}) (y)}{n}=f_{T,\alpha}^{\mathcal{C}}(x). \nonumber
\end{align}

For a given $x\in X_\infty$ let $Z_x\in \mathcal{B}_X$ with
$\mu_x(Z_x)>0$. Clearly $\mu_x(Z_x\cap W_x)=\mu_x(Z_x)>0$. For
$\delta>0$ and $\ell\in \mathbb{N}$, we define
$$Z^\ell_x(\delta)=\{y \in Z_x\cap W_x:
\mu_x(\alpha_0^{n-1}(y))\le e^{-n(f_{T,\alpha}^{\mathcal{C}}(x)-
\delta)}\ \text{for each}\ n\ge \ell\}.
$$
Then $\bigcup_{\ell=1}^{+\infty}Z^\ell_x(\delta)=Z_x\cap W_x$ by
\eqref{ee-kkk-gs}. Hence there exists $N\in \mathbb{N}$ such that
$\mu_x (Z^N_x(\delta))> 0$. For $y\in Z^N_x(\delta)$, as
$\mu_x(\alpha_0^{n-1}(y))\le e^{-n(f_{T,\alpha}^{\mathcal{C}}(x)-
\delta)}$ for each $n\ge N$, one has
\begin{equation*}
\mu_x (\alpha_0^{n- 1} (y))\le c(y)
e^{-n(f_{T,\alpha}^{\mathcal{C}}(x)- \delta)} \text{ for any }
n\in \mathbb{N},
\end{equation*}
where $c(y)=\max\{ 1,\sum_{i=1}^{N-1}
e^{i(f_{T,\alpha}^{\mathcal{C}}(x)- \delta)}\}\in (0,+\infty)$. Thus
applying Lemma \ref{non-uniform} to $Z^N_x(\delta)$ we obtain
$h^B_\mathcal{U} (T, Z^N_x(\delta))\ge
f_{T,\alpha}^{\mathcal{C}}(x)- \delta- \log M$, so $h^B_\mathcal{U}
(T, Z_x)\ge f_{T,\alpha}^{\mathcal{C}}(x)- \delta- \log M$. Note
that the last inequality is true for any $\delta>0$, one has
$h^B_\mathcal{U} (T, Z_x)\ge f_{T,\alpha}^{\mathcal{C}}(x)- \log M$.

\medskip
(2) For $k\in \mathbb{N}$ we take $\mathcal{U}_k\in \mathcal{C}_X^o$
with $||\mathcal{U}_k||\le \frac{1}{k}$. Using \cite[Lemma 2]{B3}
for each $n\in \mathbb{N}$ there exists $\alpha_{n, k}\in
\mathcal{P}_X$ such that $\alpha_{n,k}\succeq
(\mathcal{U}_k)_{0}^{n-1}$ and at most $n\# \mathcal{U}_k$ elements
of $\alpha_{n,k}$ can have a point in all their closures, here $\#
\mathcal{U}_k$ means the cardinality of $\mathcal{U}_k$. It's easy
to construct $\mathcal{U}_{n, k}\in \mathcal{C}_X^o$ such that each
element of $\mathcal{U}_{n, k}$ has a non-empty intersection with at
most $n \# \mathcal{U}_k$ elements of $\alpha_{n,k}$ (see also
\cite[Lemma 1]{B3}).

Now assume that $\mu$ is ergodic. Let
$f_{T^n,\alpha_{n,k}}^{\mathcal{C}} (x)$ be the function obtained in
Theorem \ref{SMB-conditional version} for $T^n$, $\alpha_{n,k}$ and
$\mathcal{C}$. Then $f_{T^n,\alpha_{n,k}}^{\mathcal{C}} (x)$ is
$T^n$-invariant and $\int_X f_{T^n,\alpha_{n,k}}^{\mathcal{C}} (x)
d\mu (x)= h_\mu (T^n, \alpha_{n,k}| \mathcal{C})$. Let
$g^k_n(x)=\frac{1}{n}
\sum_{i=0}^{n-1}f_{T^n,\alpha_{n,k}}^{\mathcal{C}} (T^ix)$. Then
$g^k_n(x)$ is $T$-invariant, as $f_{T^n,\alpha_{n,k}}^{\mathcal{C}}
(x)$ is $T^n$-invariant. Moreover, since $\mu\in
\mathcal{M}^e(X,T)$, $g^k_n(x)$ is constant and
\begin{eqnarray} \label{star}
g^k_n(x)&\equiv& \int_X g^k_n(y) d
\mu(y)=\frac{1}{n}\sum_{i=0}^{n-1}\int_X
f_{T^n,\alpha_{n,k}}^{\mathcal{C}} (T^iy) d
\mu(y)\nonumber \\
&=&\frac{1}{n}\sum_{i=0}^{n-1}\int_X
f_{T^n,\alpha_{n,k}}^{\mathcal{C}} (y) d \mu(y)= h_\mu (T^n,
\alpha_{n,k}| \mathcal{C})
\end{eqnarray}
 for $\mu$-a.e. $x\in X$. By (1) for $\mu$-a.e. $x\in X$, if
 $Z_x\in \mathcal{B}_X$ with $\mu_x(Z_x)>0$ then
\begin{equation}\label{eq-all-1}
h^B_{\mathcal{U}_{n,k}} (T^n, Z_x)\ge
f_{T^n,\alpha_{n,k}}^{\mathcal{C}}(x)- \log (n \# \mathcal{U}_k)
\text{ for each }k\in \mathbb{N},\ n\in \mathbb{N}.
\end{equation}
Moreover, note that $T\mu_x=\mu_{Tx}$ for $\mu$-a.e. $x\in X$, there
exists a $T$-invariant subset $X_1\subseteq X$ with $\mu(X_1)=1$
such that $T\mu_x=\mu_{Tx}$ and both \eqref{star} and
\eqref{eq-all-1} hold for all $x\in X_1$.

Now for any given $x\in X_1$ let $Z_x\in \mathcal{B}_X$ with
$\mu_x(Z_x)>0$. Then $T^ix\in X_1$ and $\mu_{T^ix}(T^i
Z_x)=T^i\mu_x(T^iZ_x)=\mu_x(Z_x)>0$ for any $i\ge 0$. By
\eqref{eq-all-1}, for each $k\in \mathbb{N}$, $n\in \mathbb{N}$ and
$i\ge 0$,
$$h^B(T^n, T^iZ_x)\ge h^B_{\mathcal{U}_{n,k}} (T^n, T^iZ_x)\ge
f_{T^n,\alpha_{n,k}}^{\mathcal{C}}(T^ix)- \log (n \#
\mathcal{U}_k).
$$
For each $n\in \mathbb{N}$, as $h^B(T^n,Z_x)\ge h^B(T^n,T^iZ_x)$
for each $i\ge 0$ (see Proposition \ref{06.02.28} (3)), we have
\begin{eqnarray*}
h^B(T^n,Z_x)&\ge &   \frac{1}{n}\sum_{i=0}^{n-1} h^B(T^n,T^iZ_x)
\ge \frac{1}{n}\sum_{i=0}^{n-1} \left (
f_{T^n,\alpha_{n,k}}^{\mathcal{C}}(T^ix)- \log (n \#
\mathcal{U}_k) \right)\\
&=& g^k_n(x)-\log (n \# \mathcal{U}_k)=h_\mu (T^n, \alpha_{n,k}|
\mathcal{C})-\log (n \# \mathcal{U}_k)
\end{eqnarray*}
for all $k\in \mathbb{N}$. Then using Proposition \ref{06.02.28}
(4) we have
\begin{eqnarray*}
h^B (T, Z_x)&=& \frac{h^B (T^n, Z_x)}{n}\ge \frac{1}{n}
(h_\mu (T^n,
\alpha_{n,k}| \mathcal{C})-\log (n \#
\mathcal{U}_k))\\
&\ge & \frac{1}{n} ( h_\mu (T^n, (\mathcal{U}_k)_{0}^{n-1}|
\mathcal{C})-\log (n \# \mathcal{U}_k))=h_\mu (T, \mathcal{U}_k|
\mathcal{C})-\frac{\log (n \# \mathcal{U}_k)}{n}
\end{eqnarray*}
for each $k\in \mathbb{N}$ and $n\in \mathbb{N}$. Now fixing $k\in
\mathbb{N}$ letting $n\rightarrow +\infty$ we get $h^B (T, Z_x)\ge
h_\mu (T, \mathcal{U}_k| \mathcal{C})$. Finally letting
$k\rightarrow +\infty$ we have $h^B (T, Z_x)\ge \lim_{k\rightarrow
+\infty}h_\mu (T, \mathcal{U}_k|
\mathcal{C})=h_\mu(T,X|\mathcal{C})$. This completes the proof of
(2) since $\mu(X_1)=1$.
\end{proof}

The following result is an application of Proposition \ref{more
estimate-not measurable}.

\begin{lem} \label{lemma for cover}
Let $(X, T)$ be a TDS, $\mu\in \mathcal{M}^e(X, T)$ and
$\mathcal{C}\subseteq \mathcal{B}_\mu$ a $T$-invariant
sub-$\sigma$-algebra. If $\mu= \int_X \mu_x d \mu (x)$ is the
disintegration of $\mu$ over $\mathcal{C}$, then
\begin{enumerate}

\item If $\alpha\in \mathcal{P}_X$ then for $\mu$-a.e. $x\in
X$, fixing each $x$, for each $\epsilon\in (0,1)$ there exists a
compact subset $Z_x (\alpha, \epsilon)$ of $X$ such that $\mu_x
(Z_x (\alpha,\epsilon))\ge 1- \epsilon$ and
\begin{equation*}
h^B_\alpha (T, Z_x (\alpha, \epsilon))= h_\alpha (T, Z_x (\alpha,
\epsilon))= h_\mu (T, \alpha| \mathcal{C}).
\end{equation*}

\item For $\mu$-a.e. $x\in
X$, fixing each $x$, for each $\epsilon\in (0,1)$ there exists a
compact subset $Z_x (\epsilon)$ of $X$ such that $\mu_x (Z_x
(\epsilon))\ge 1- \epsilon$ and
\begin{equation*}
h^B (T, Z_x (\epsilon))= h (T, Z_x (\epsilon))= h_\mu (T, X|
\mathcal{C}).
\end{equation*}
\end{enumerate}
\end{lem}
\begin{proof}
(1) Let $\alpha\in \mathcal{P}_X$. As $\mu\in \mathcal{M}^e (X, T)$,
by \eqref{ee-kkk-gs} there exists $X_\infty\in \mathcal{B}_X$ with
$\mu (X_\infty)= 1$ such that for each $x\in X_\infty$,  one can
find $W_x\in \mathcal{B}_X$ with $\mu_x(W_x)=1$ such that for each
$y\in W_x$
\begin{equation*}
\lim_{n\rightarrow +\infty} - \frac{1}{n} \log \mu_x (\alpha_0^{n-
1} (y))= h_\mu (T, \alpha| \mathcal{C})
\end{equation*}
(for details see the proof of Proposition \ref{more estimate-not
measurable} (1)). By Proposition \ref{more estimate-not measurable}
(1), w.l.g. we may require
\begin{equation}\label{req-11}
h^B_\alpha (T, Z)\ge  h_\mu (T, \alpha| \mathcal{C})
\end{equation}
for any $x\in X_\infty$ and $Z\in \mathcal{B}_X$ with $\mu_x(Z)>0$
(if necessary we take a  subset of $X_\infty$).

Let $x\in X_\infty$.  Obviously, for each $y\in W_x$ and $m\in
\mathbb{N}$ there exists $n_{y,m}\in \mathbb{N}$ such that if $n\ge
n_{y,m}$ then $- \frac{1}{n} \log \mu_x (\alpha_0^{n- 1} (y))\le
h_\mu (T, \alpha| \mathcal{C})+ \frac{1}{m}$, i.e. $\mu_x
(\alpha_0^{n- 1} (y))\ge e^{- n (h_\mu (T, \alpha| \mathcal{C})+
\frac{1}{m})}$. So for each $n\in \mathbb{N}$,
\begin{equation*}
 \mu_x (\alpha_0^{n- 1} (y))\ge
e^{- (n+ n_{y,m}) (h_\mu (T, \alpha| \mathcal{C})+ \frac{1}{m})}.
\end{equation*}
 We introduce a $\mu_x$-measurable function by defining for
$\mu_x$-a.e. $y\in X$
\begin{equation*}
c_m(y)= \inf_{n\in \mathbb{N}} \frac{\mu_x (\alpha_0^{n- 1}
(y))}{e^{- n (h_\mu (T, \alpha| \mathcal{C})+ \frac{1}{m})}}\ge
e^{- n_{y,m} (h_\mu (T, \alpha| \mathcal{C})+ \frac{1}{m})}> 0.
\end{equation*}
Let $Z_{k}^m= \{y\in W_x: c_m(y)\ge \frac{1}{k}\}$ for each $k\in
\mathbb{N}$. Then $Z_k^m$ is $\mu_x$-measurable and $h_\alpha (T,
Z_{k}^m)\le h_\mu (T, \alpha| \mathcal{C})+ \frac{1}{m}$ by Lemma
\ref{uniform}. Moreover, as $\lim_{k\rightarrow +\infty}\mu_x
(Z_{k}^m)=1$,  for each $\epsilon\in (0,1)$ there exists a compact
subset $B_{\epsilon}^m\subseteq Z_{K}^m$ for some $K\in \mathbb{N}$
such that $\mu_x (X\setminus B_\epsilon^m)< \frac{\epsilon}{2^m}$
and $h_\alpha (T, B_\epsilon^m)\le h_\mu (T, \alpha| \mathcal{C})+
\frac{1}{m}$.

For $\epsilon\in (0,1)$, set $Z_x (\alpha, \epsilon)=
\bigcap_{m\in \mathbb{N}} B_\epsilon^m$. Then $Z_x (\alpha,
\epsilon)$ is a compact subset of $X$,
$$\mu_x (Z_x (\alpha, \epsilon))=1-
\mu \left(\bigcup_{m\in \mathbb{N}} X\setminus
B_\epsilon^m\right)\ge 1-\sum_{m\in \mathbb{N}}\mu(X\setminus
B_\epsilon^m )\ge 1-\epsilon>0
$$ and
$$\ h_\alpha (T, Z_x (\alpha, \epsilon))\le \inf_{m\in
\mathbb{N}}h_\alpha (T, B_\epsilon^m)\le \inf_{m\in \mathbb{N}}
\left(h_\mu (T, \alpha| \mathcal{C})+\frac{1}{m}\right)=
h_\mu(T,\alpha|\mathcal{C}).$$ Moreover, using Lemma \ref{bridge}
and \eqref{req-11} we have
$$h^B_\alpha (T, Z_x (\alpha, \epsilon))= h_\alpha (T, Z_x (\alpha,
\epsilon))= h_\mu (T, \alpha| \mathcal{C}).$$

(2) Let $\{\mathcal{U}_n\}_{n\in \mathbb{N}}\subseteq
\mathcal{C}_X^o$ with $\lim_{n\rightarrow +\infty}
||\mathcal{U}_n||=0$. For $n\in \mathbb{N}$ we take $\alpha_n\in
\mathcal{P}_X$ with $\alpha_n\succeq \mathcal{U}_n$. By (1) there
exists a measurable subset $X'$ of $X$ with $\mu(X')=1$ such that if
$x\in X'$ then
 for each $\epsilon\in (0,1)$ and $n\in \mathbb{N}$ there exists
a compact subset $Z_{x}(n,\epsilon)$ such that
\begin{equation*}
\mu_x (Z_{x}(n,\epsilon))\ge 1- \frac{\epsilon}{2^n}\ \text{ and
}\ h_{\mathcal{U}_n} (T, Z_{x}(n,\epsilon))\le h_{\alpha_n} (T,
Z_{x}(n,\epsilon))=h_\mu (T, \alpha_n| \mathcal{C}).
\end{equation*}
By Proposition \ref{more estimate-not measurable} (2), w.l.g. (if
necessary we take a subset of $X'$) we may require
\begin{equation}\label{req-12}
h^B(T, Z)\ge  h_\mu (T, X| \mathcal{C})
\end{equation}
for any $x\in X'$ and $Z\in \mathcal{B}_X$ with $\mu_x(Z)>0$.

Let $x\in X'$. For $\epsilon\in (0,1)$, Set $Z_x (\epsilon)=
\bigcap_{n\in \mathbb{N}} Z_x(n,\epsilon)$. Then $Z_x (\epsilon)$ is
a compact subset of $X$,
$$\mu_x (Z_x (\epsilon))=1-
\mu \left(\bigcup_{n\in \mathbb{N}} X\setminus Z_x(n,\epsilon)
\right)\ge 1-\sum_{n\in \mathbb{N}}\mu(X\setminus
Z_x(n,\epsilon))\ge 1-\epsilon>0
$$ and
\begin{align*}
h (T, Z_x (\epsilon))&=\sup_{n\in
\mathbb{N}}h_{\mathcal{U}_n}(T,Z_x (\epsilon) )\le \sup_{n\in
\mathbb{N}}h_{\mathcal{U}_n}(T,Z_x (n,\epsilon) )\\
&\le \sup_{n\in \mathbb{N}} h_\mu (T, \alpha_n| \mathcal{C})\le
h_\mu(T,X|\mathcal{C}).
\end{align*}
 Moreover, using Lemma \ref{bridge} and
\eqref{req-12} we have if $x\in X'$ then $h^B (T, Z_x (\epsilon))= h
(T, Z_x (\epsilon))= h_\mu (T, X| \mathcal{C})$. This finishes the
proof of (2) since $\mu (X')= 1$.
\end{proof}

With the above preparations we can obtain the main result of this
section.

\begin{thm} \label{aim of paper}
Let $(X, T)$ be a TDS with finite entropy. Then for each $0\le h\le
h_{\text{top}} (T, X)$ there exists a non-empty compact subset $K_h$
of $X$ such that $h^B (T, K_h)= h (T, K_h)= h$. In particular, $(X,
T)$ is lowerable.
\end{thm}
\begin{proof} If $h=h_{\text{top}} (T, X)$, it is true for $K_h=X$
by Proposition \ref{06.02.28} (1). If $h=0$, it is true for $K_h=\{
x\}$ for any $x\in X$. Now we assume $0< h< h_{\text{top}} (T, X)$.
By the variational principle there exists $\mu\in \mathcal{M}^e (X,
T)$ with $h<h_\mu(T)\le h_{\text{top}} (T, X)<+ \infty$. It is well
known \cite[Theorem 15.11]{G} that there exists a $T$-invariant
sub-$\sigma$-algebra $\mathcal{C}\subseteq \mathcal{B}_\mu$ such
that $h_\mu (T, X| \mathcal{C})= h$, where $\mathcal{B}_\mu$ is the
completion of $\mathcal{B}_X$ under $\mu$. Then the conclusion
follows from Lemma \ref{lemma for cover} (2).
\end{proof}

%\begin{cor} Let $(X,T)$ be a TDS with zero mean topological
%dimension. Then for each $0\le h\le h_{\text{top}} (T, X)$ there are
%a factor $(Y,S)$ of $(X,T)$ and a compact $K_h\subset Y$ with
%$h(T,K_h)=h$.
%\end{cor}

Let $(X, T)$ be a TDS, $\mu\in \mathcal{M}(X, T)$ and
$\mathcal{C}\subseteq \mathcal{B}_\mu$ a $T$-invariant
sub-$\sigma$-algebra with $\mu= \int_X \mu_x d \mu (x)$ the
disintegration of $\mu$ over $\mathcal{C}$. For $\mu$-a.e. $x\in
X$,  we define
\begin{equation*}
h^B_\mathcal{U} (T, \mu, x)= \inf \{h^B_\mathcal{U} (T, Z): Z\in
\mathcal{B}_X \text{ with } \mu_x (Z)= 1\}
\end{equation*}
for any given $\mathcal{U}\in \mathcal{C}_X$ and
\begin{equation*}
h^B (T, \mu, x)= \sup_{\mathcal{U}\in \mathcal{C}_X^o}
h^B_\mathcal{U} (T, \mu, x).
\end{equation*}
The {\it essential supremum} of a real valued function $f$ defined
on a subset of $X$ with $\mu$-full measure is defined by
$$\mu-\sup f(x)=\inf_{\mu(X')= 1} \sup_{x\in X'}f(x).$$
We are not sure of the $\mu$-measurability of functions both
$h^B_\mathcal{U} (T, \mu, x)$ and $h^B (T, \mu, x)$ w.r.t. $x\in X$.
Whereas, using Proposition \ref{more estimate-not measurable} we
have the following result.

\begin{cor} \label{estimate-not measurable}
Let $(X, T)$ be a TDS, $\mu\in \mathcal{M}(X, T)$ and
$\mathcal{C}\subseteq \mathcal{B}_\mu$ a $T$-invariant
sub-$\sigma$-algebra. If $\mu= \int_X \mu_x d \mu (x)$ is the
disintegration of $\mu$ over $\mathcal{C}$, then
\begin{enumerate}

\item Let $\mathcal{U}\in
\mathcal{C}_X$ and $\alpha\in \mathcal{P}_X$ with
$f_{T,\alpha}^{\mathcal{C}}(x)$ the function obtained in Theorem
\ref{SMB-conditional version} for $T$, $\alpha\in \mathcal{P}_X$
and $\mathcal{C}$. Assume that each element of $\mathcal{U}$ has a
non-empty intersection with at most $M$ elements of $\alpha$
($M\in \mathbb{N}$). Then for $\mu$-a.e. $x\in X$,
$$h^B_\mathcal{U} (T, \mu, x)\ge f_{T,\alpha}^{\mathcal{C}}(x)- \log M$$ and if $\mu$ is
ergodic then $$h^B_\mathcal{U} (T, \mu, x)\ge h_\mu (T, \alpha|
\mathcal{C})- \log M.$$

\item $\mu-\sup h^B (T, \mu, x)\ge h_\mu (T, X|
\mathcal{C})$; moreover, if $\mu$ is ergodic then $h^B (T, \mu,
x)=h_\mu (T, X| \mathcal{C})$ for $\mu$-a.e. $x\in X$.
\end{enumerate}
\end{cor}
\begin{proof}
(1) is just a direct corollary of Proposition \ref{more estimate-not
measurable} (1).

(2) For $k\in \mathbb{N}$ we take $\mathcal{U}_k\in \mathcal{C}_X^o$
with $||\mathcal{U}_k||\le \frac{1}{k}$. Then, for each $n\in
\mathbb{N}$, we take $\alpha_{n, k}\in \mathcal{P}_X$ such that
$\alpha_{n,k}\succeq (\mathcal{U}_k)_{0}^{n-1}$ and at most $n\#
\mathcal{U}_k$ elements of $\alpha_{n,k}$ can have a point in all
their closures, and take $\mathcal{U}_{n, k}\in \mathcal{C}_X^o$
such that each element of $\mathcal{U}_{n, k}$ has a non-empty
intersection with at most $n \# \mathcal{U}_k$ elements of
$\alpha_{n,k}$.

Let $f_{T^n,\alpha_{n,k}}^{\mathcal{C}} (x)$ be the function
obtained in Theorem \ref{SMB-conditional version} for $T^n$,
$\alpha_{n,k}$ and $\mathcal{C}$. Then
$f_{T^n,\alpha_{n,k}}^{\mathcal{C}} (x)$ is $T^n$-invariant and
$\int_X f_{T^n,\alpha_{n,k}}^{\mathcal{C}} (x) d\mu (x)= h_\mu (T^n,
\alpha_{n,k}| \mathcal{C})$. Then using (1), for $\mu$-a.e. $x\in
X$,
\begin{eqnarray*}
&& n h^B_{\mathcal{U}_{n, k}} (T, \mu, x) =  \inf \{n
h^B_{\mathcal{U}_{n, k}} (T, Z):  Z\in \mathcal{B}_X \text{ with }
\mu_x (Z)= 1\}
\\
&\ge & \inf \{h^B_{\mathcal{U}_{n, k}} (T^n, Z): Z\in
\mathcal{B}_X \text{ with } \mu_x (Z)= 1\}
\ (\text{by Proposition \ref{06.02.28} (4)}) \\
&\ge & f_{T^n,\alpha_{n,k}}^{\mathcal{C}} (x)- \log (n \#
\mathcal{U}_k)\ (\text{using (1)}).
\end{eqnarray*}
Hence,
\begin{eqnarray*}
& & \mu-\sup h^B (T, \mu, x)\ge \mu-\sup h^B_{\mathcal{U}_{n, k}}
(T, \mu, x) \\ & \ge & \frac{1}{n} \int_X \left(
f_{T^n,\alpha_{n,k}}^{\mathcal{C}} (x)- \log (n \# \mathcal{U}_k)
\right) d \mu(x) =\frac{1}{n}\left(  h_\mu (T^n, \alpha_{n,k}|
\mathcal{C})-\log (n \# \mathcal{U}_k) \right)\\
&\ge & \frac{1}{n}(h_\mu (T^n, (\mathcal{U}_k)_{0}^{n-1}|
\mathcal{C})-\log (n \# \mathcal{U}_k))=h_\mu (T, \mathcal{U}_k|
\mathcal{C})-\frac{1}{n}\log (n \# \mathcal{U}_k).
\end{eqnarray*}
Fixing $k\in \mathbb{N}$ letting $n\rightarrow +\infty$ in the above
inequality we obtain $\mu-\sup h^B (T, \mu, x)\ge h_\mu (T,
\mathcal{U}_k| \mathcal{C})$. Then letting $k\rightarrow +\infty$ we
obtain $\mu-\sup h^B (T, \mu, x)\ge h_\mu (T, X| \mathcal{C})$.

Now we assume that $\mu$ is ergodic. First, by Proposition \ref{more
estimate-not measurable} (2), we know $h^B (T, \mu, x)$ $\ge h_\mu
(T, X| \mathcal{C})$ for $\mu$-a.e. $x\in X$. Secondly using Lemma
\ref{lemma for cover} (2),  there exists a measurable subset $X'$ of
$X$ with $\mu(X')=1$ such that if $x\in X'$ then
 for each $\ell\in \mathbb{N}$ there exists a
compact subset $Z_x (\ell)$ of $X$ such that $\mu_x (Z_x
(\epsilon))\ge 1- \frac{1}{2^{\ell}}$ and $h^B (T, Z_x
(\ell))=h_\mu (T, X| \mathcal{C})$. Next for each $x\in X'$, let
$Z_x=\bigcup_{\ell\in \mathbb{N}}Z_x(\ell)$. Then $Z_x\in
\mathcal{B}_X$ with $\mu(Z_x)=1$ and $h^B(T,Z_x)=\sup_{\ell\in
\mathbb{N}}h^B(T,Z_x(\ell))=h_\mu (T, X| \mathcal{C})$. This
implies $h^B (T, \mu, x)\le h_\mu (T, X| \mathcal{C})$. Collecting
terms, $h^B (T, \mu, x)=h_\mu (T, X| \mathcal{C})$ for $\mu$-a.e.
$x\in X$.
\end{proof}

Following from the proof of Corollary \ref{estimate-not measurable},
we are easy to show the following result.

\begin{cor} \label{lemma for Bowen}
Let $(X, T)$ be a TDS, $\mu\in \mathcal{M}^e(X, T)$ and
$\mathcal{C}\subseteq \mathcal{B}_\mu$ a $T$-invariant
sub-$\sigma$-algebra. If $\mu= \int_X \mu_x d \mu (x)$ is the
disintegration of $\mu$ over $\mathcal{C}$, then
\begin{enumerate}

\item If $\alpha\in \mathcal{P}_X$ then for $\mu$-a.e. $x\in
X$ there exists $Z_x\in \mathcal{B}_X$ such that $\mu_x (Z_x)= 1$
and $h^B_\alpha (T, Z_x)$ $= h_\mu (T, \alpha| \mathcal{C})$.
Moreover, $h_\alpha^B (T, \mu, x)= h_\mu (T, \alpha| \mathcal{C})$
for $\mu$-a.e. $x\in X$.

\item For $\mu$-a.e. $x\in
X$ there exists $Z_x\in \mathcal{B}_X$ such that $\mu_x (Z_x)= 1$
and $h^B (T, Z_x)= h_\mu (T, X| \mathcal{C})$.
\end{enumerate}
\end{cor}
\begin{rem} We can't expect similar results hold for topological
entropy of subsets using open covers. For example, let $(X, T)$ be a
minimal TDS, $\mu\in \mathcal{M}^e (X, T)$ and $\mathcal{C}=
\{\emptyset, X\}$ such that $0<h_\mu(T, X)<h_{\text{top}}(T,X)$. Let
$\mu=\int_X \mu_x d\mu(x)$ be the disintegration of $\mu$ over
$\mathcal{C}$, then $\mu_x=\mu$ for $\mu$-a.e. $x\in X$. Thus for
$\mu$-a.e. $x\in X$, if $Z\in \mathcal{B}_X$ with $\mu_x (Z)= 1$
then $\overline{Z}= X$, which implies $h(T, Z)=h (T, \overline{Z})=
h_{\text{top}} (T, X)>h_\mu(T, X)= h_\mu(T, X|\mathcal{C})$.
\end{rem}

\section{Expansive cases}

In this section by direct construction we shall prove that each
expansive TDS is \HUL. Recall that we say a TDS $(X, T)$ is {\it
expansive} if there exists $\delta> 0$ such that $x\neq y$ implies
$\sup_{n\in \mathbb{Z}} d (T^n x, T^n y)> \delta$. In this case,
$\delta$ is called an {\it expansive constant}. In particular, each
symbolic TDS is expansive.

\medskip

To do this let's first recall \cite[Remark 5.13]{YZ}. Let $(X, T)$
be a TDS with metric $d$ and $E$ a compact subset. For each
$\epsilon> 0$ and $x\in E$ we define $$h_d (x, \epsilon, E)= \inf
\{r(d, T, \epsilon, K): K\ \text{is a compact neighborhood of}\ x\
\text{in}\ E\}.$$ Let $h (x, E)= \lim_{\epsilon\rightarrow 0 +}
h_d (x, \epsilon, E)$. Its value depends only on the topology on
$X$. The following is \cite[Remark 5.13]{YZ}.

\begin{thm} \label{ye-zhang}
Let $(X,T)$ be a TDS with metric $d$ and $E$ a compact subset. Then
\begin{enumerate}

\item $h_d (x, \epsilon, E)$ is u.s.c. on $E$ and $\sup_{x\in
E}h(x,E)=h(T,E)$.

\item For each $x\in E$ there is a countable compact subset
$E_x\subseteq E$ with a unique limit point $x$ such that $h (T,
E_x)= h (x, E)$.

\item  There is a countable compact subset $E'\subseteq E$ with
$h(T,E')=h(T,E)$. Moreover, $E'$ can be chosen such that the set
of its limit points has at most one limit point, and $E'$ has a
unique limit point iff there is $x\in E$ with $h(x,E)=h(T,E)$.
\end{enumerate}
\end{thm}

The first result is the following lemma.

\begin{lem} \label{good}
Let $(X, T)$ be a TDS with metric $d$ and $K\subseteq X$ a compact
subset with $h(T,K)>0$. Then for any $0< h< h (T, K)$ there is a
$\delta_0> 0$ such that if $0< \delta\le \delta_0$ then there is a
countable compact subset $K_{h, \delta}\subseteq K$ with a unique
limit point such that $s (d, T, \delta, K_{h, \delta})= h$.
\end{lem}
\begin{proof}
Let $0< h< h(T, K)$. By Theorem \ref{ye-zhang} there exists a
countable compact subset $K_0\subseteq K$ with a unique limit point
$x_0$ such that $h (T, K_0)> h$, thus for some $\delta_0> 0$ if $0<
\delta\le \delta_0$ then $s (d, T, \delta, K_0)> h$. Now let $0<
\delta\le \delta_0$ be fixed.

Define $l_1$ to be the minimal integer $n\in \mathbb{N}$ such that
\begin{equation*}
\exists B_1\subseteq K_0 \ \text{is} \ (n, \delta)\text{-separated}
\ \text{w.r.t.} \ T \ \text{s.t.} \ |B_1|= [e^{n h}] + 2,
\end{equation*}
here $|B_1|$ means the cardinality of $B_1$. It is clear that $l_1$
is finite, as $s (d, T, \delta, K_0)> h$. Let $A_1= D_1\subseteq
K_0$ be $(l_1, \delta)$-separated w.r.t. $T$ with $|A_1|= [e^{l_1
h}] + 1$ and $x_0\notin A_1$.

Define $l_2$ to be the minimal integer $n> l_1$ such that
\begin{equation*}
\exists B_2\subseteq (B_{d_{l_1}} (x_0, \delta)\cap K_0)\setminus
A_1 \ \text{is} \ (n, \delta)\text{-separated} \ \text{w.r.t.} \ T \
\text{s.t.} \ |B_2|= [e^{n h}]- [e^{l_1 h}] + 2,
\end{equation*}
where $B_{d_{l_1}} (x_0, \delta)$ denotes the open ball with center
$x_0$ and radius $\delta$ (in the sense of $d_{l_1}$-metric). Since
$x_0$ is the unique limit point of the countable compact subset
$K_0\subseteq K$, $K_0\setminus ((B_{d_{l_1}} (x_0, \delta)\cap
K_0)\setminus A_1)$ is a finite subset, so
\begin{equation*}
s (d, T, \delta, (B_{d_{l_1}} (x_0, \delta)\cap K_0)\setminus A_1)=
s (d, T, \delta, K_0)> h,
\end{equation*}
which implies that $l_2> l_1$ is finite. Let $D_2\subseteq
(B_{d_{l_1}} (x_0, \delta)\cap K_0)\setminus A_1$ be $(l_2,
\delta)$-separated w.r.t. $T$ with $|D_2|= [e^{l_2 h}]- [e^{l_1 h}]
+ 1$ and $x_0\notin D_2$. Set $A_2= A_1\cup D_2\not \ni x_0$. Then
$|A_2|= [e^{l_2 h}]+ 2$ and $A_2\subseteq K_0$ is $(l_2,
\delta)$-separated w.r.t. $T$.

By induction there are $l_1< l_2< \cdots$ and $A_1\subseteq
A_2\subseteq \cdots\subseteq K_0$ such that for each $i\in \N$

(1) $x_0\not \in A_i$ and $A_{i+ 1}\setminus A_i\subseteq
B_{d_{l_i}} (x_0, \delta)\cap K_0$.

(2) $A_i$ is $(l_i, \delta)$-separated w.r.t. $T$ and $|A_i|=
[e^{l_i h}]+ i$.

Set $A_\infty = \{x_0\}\cup \bigcup_{i\ge 1} A_i\subseteq K_0$. Then
$A_\infty (\subseteq K_0)\subseteq K$ is a countable compact subset
with $x_0$ as its unique limit point in $X$. If $l_{n}\le l< l_{n+
1}$ then let $A\subseteq A_\infty$ be $(l, \delta)$-separated w.r.t.
$T$. As $A\setminus A_n\subseteq (B_{d_{l_n}} (x_0, \delta)\cap
K_0)\setminus A_n$ is $(l, \delta)$-separated w.r.t. $T$, because of
the definition of $l_{n+ 1}$ we have $|A\setminus A_n|\le [e^{l h}]-
[e^{l_n h}]+ 1$, which implies
\begin{equation*}
|A|\le |A\setminus A_n|+ |A_n|\le ([e^{l h}]- [e^{l_n h}]+ 1)+
([e^{l_n h}]+ n)= [e^{l h}]+ n+ 1.
\end{equation*}
Then $s_{l} (d, T, \delta, A_\infty)\le [e^{l h}]+ n+ 1$ for all
$l_{n}\le l< l_{n+ 1}$. Note that $s_{l_n} (d, T, \delta,
A_\infty)\ge [e^{l_n h}]+ n$, we conclude $s(d, T, \delta, A_\infty)
= h$. Take $K_{h, \delta}=A_\infty$. This completes the proof.
\end{proof}

Before proving that each expansive TDS is \HUL\ we need the
following result.

\begin{lem}\label{huang}
Let $(X, T)$ be an expansive TDS with metric $d$ and an expansive
constant $\delta> 0$. Then for any compact subset $K\subseteq X$, $h
(T, K)= s (d, T, \frac{\delta}{2}, K)$.
\end{lem}
\begin{proof}
Let $K\subseteq X$ be a compact subset and $\epsilon>0$. We claim
that there exists $n(\epsilon)\in \mathbb{N}$ such that for $x,y\in
X$, $d^*_{n(\epsilon)}(x,y)\le \frac{\delta}{2}$ implies $d(x,y)\le
\epsilon$, where
$d_{n(\epsilon)}^*(x,y)=\max_{i=-n(\epsilon)}^{n(\epsilon)}
d(T^ix,T^iy)$. In fact, if it is not the case, then for each $n\in
\mathbb{N}$ there exist $x_n,y_n\in X$ such that $d^*_n(x_n,y_n)\le
\frac{\delta}{2}$ and $d(x_n,y_n)>\epsilon$. W.l.g. we assume
$\lim_{n\rightarrow +\infty} (x_n,y_n)=(x,y)$. Then $d(x,y)\ge
\epsilon$ and $d^*_m(x,y)\le \frac{\delta}{2}$ for each $m\in
\mathbb{N}$, which contradicts that $\delta$ is an expansive
constant of $(X,T)$.

Now for each $m\in \mathbb{N}$, let $E$ be an
$(m,\epsilon)$-separated subset of $K$ w.r.t. $T$, then
$T^{-n(\epsilon)}E$ is an $(m+2n(\epsilon),
\frac{\delta}{2})$-separated subset of $T^{-n(\epsilon)}K$ w.r.t.
$T$. Hence
\begin{eqnarray*}
s (d, T, \epsilon, K)&= & \limsup_{m\rightarrow +\infty}
\frac{1}{m} \log s_m (d, T, \epsilon, K) \\
&\le & \limsup_{m\rightarrow +\infty} \frac{1}{m} \log s_{m+ 2 n
(\epsilon)} \left(d, T, \frac{\delta}{2}, T^{- n (\epsilon)} K\right) \\
&= & s \left(d, T, \frac{\delta}{2}, T^{- n (\epsilon)} K\right)= s
\left(d, T, \frac{\delta}{2}, K\right).
\end{eqnarray*}
Sine $\epsilon> 0$ is arbitrary, letting $\epsilon\rightarrow 0+$ we
conclude $h (T, K)= s (d, T, \frac{\delta}{2}, K)$.
\end{proof}

Now we are ready to prove the main result in this section.

\begin{thm} \label{expansive}
Each expansive TDS is \HUL.
\end{thm}
\begin{proof}
Let $(X, T)$ be an expansive TDS with metric $d$ and an expansive
constant $2 \delta> 0$.

Let $E\subseteq X$ be a non-empty compact subset and $0\le h\le h
(T, E)$. If $h= 0$ then $h (T, \{x\})= 0$ for any $x\in E$. Now
assume $h=h (T, E)>0$, then by Theorem \ref{ye-zhang} (1), $h_d (x,
\delta, E)$ is u.s.c. on $E$ and $h (T, E)= \sup_{x\in X} h (x, E)$.
Note that
\begin{eqnarray*}
h (x, E)&= & \lim_{\epsilon\rightarrow 0 +} \inf \{r (d, T,
\epsilon, K): K\ \text{is a compact neighborhood of}\ x\ \text{in}\
E\} \\
&= & \lim_{\epsilon\rightarrow 0 +} \inf \{s (d, T, \epsilon, K): K\
\text{is a compact neighborhood of}\ x\ \text{in}\
E\} \\
&= & h_d (x,\delta, E)\ (\text{using Lemma \ref{huang}}),
\end{eqnarray*}
then $h (T, E)= \max_{x\in E} h_d (x, \delta, E)$. Say $x_0\in E$
with $h_d (x_0, \delta, E) (= h (x_0, E)) = h (T, E)>0$. By Theorem
\ref{ye-zhang} (2) there exists a countable compact subset
$K_0\subseteq E$ with $x_0$ as its unique limit point such that $h
(T, K_0)= h (x_0, E) = h (T, E)$. This completes the proof in the
case of $h= h (T, E)>0$.

Now assume $0< h< h (T, E)$. By Lemma \ref{good} there exist $0<
\epsilon\le \delta$ small enough and a countable compact subset
$K_h\subseteq E$ with a unique limit point such that $ s(d, T, \ep,
K_h)= h$. Since $2 \epsilon$ is also an expansive constant, $h (T,
K_h)= h$ by Lemma \ref{huang}.

Thus for each non-empty compact subset $E$ and each $0\le h\le h (T,
E)$ there is a non-empty compact subset $K_h\subseteq K$ with a
unique limit point such that $h (T, K_h)= h$, that is, TDS $(X, T)$
is \HUL.
\end{proof}

 Note that we can also introduce the \HUL\
property for a GTDS, the results of this section remain true for a
GTDS. In particular, following similar discussions, the results in
Lemma \ref{huang} and Theorem \ref{expansive} also hold for each
positively expansive dynamical system. Let $X$ be a compact metric
space endowed with a continuous surjection $T: X\rightarrow X$ and a
compatible metric $d$. Recall that we say $(X, T)$ {\it positively
expansive} if there exists $\delta> 0$ such that $x\neq y$ implies
$\sup_{n\in \mathbb{Z}_+} d (T^n x, T^n y)> \delta$ (see for example
\cite{RW}). We also call $\delta$ an {\it expansive constant}.

\section{A \HUL\ TDS is asymptotically h-expansive}

In this section we shall answer Question \ref{q 3} partially. Note
that the invertibility can be removed for TDSs considered in this
section without changing our results.

\medskip

We discuss two classes of weak expansiveness: the
$h$-expansiveness and asymptotical $h$-expansiveness, introduced
by Bowen \cite{Bex} and Misiurewicz  \cite{Mi}, respectively. Let
$(X, T)$ be a GTDS with metric $d$. For each $\epsilon> 0$ we
define
\begin{equation*}
h_T^*(\epsilon)= \sup_{x\in X} h (T, \Phi_{\epsilon} (x)), \
\text{where}\ \Phi_{\epsilon} (x)= \{y\in X: d(T^n x, T^n y)\le
\epsilon\ \text{if}\ n\ge 0\}.
\end{equation*}
$(X, T)$ is called {\it $h$-expansive} if there exists an $\epsilon>
0$ such that $h_T^*(\epsilon)= 0$, and is called {\it asymptotically
$h$-expansive} if $\lim_{\epsilon\rightarrow 0+} h_T^*(\epsilon)=0$.
It is shown by Bowen \cite{Bex} that positively expansive systems,
expansive homeomorphisms, endomorphisms of a compact Lie group and
Axiom $A$ diffeomorphisms are all $h$-expansive, by Misiurewicz
\cite{Mi1} that every continuous endomorphism of a compact metric
group is asymptotically $h$-expansive if it has finite entropy, and
by Buzzi \cite{Bu} that any $C^\infty$ diffeomorphism on a compact
manifold is asymptotically $h$-expansive.

In this section we prove that each \HUL\ TDS is asymptotically
$h$-expansive.

\medskip

The following two results seem too technical but interesting
themselves, which are needed in proving the main result of this
section.

\begin{thm}\label{zero}
Let $(X,T)$ be a TDS. Then for any compact subset $K\subseteq X$
with $h (T, K)> 0$, there is a countable infinite compact subset
$K_\infty\subseteq K$ such that $h (T, K_\infty)=0$.
\end{thm}
\begin{proof}
First, there is a countable compact subset $K_0= \{x, x_1, x_2,
\cdots\}\subseteq K$ such that $h= h (T, K_0)> 0$ and
$\lim_{n\rightarrow +\infty} x_n= x$ (using Theorem \ref{ye-zhang}).

Let $d$ be a metric on $(X, T)$. For sufficiently small $\ep_1> 0$
let $K_1\subseteq K_0$ be the subset constructed in Lemma \ref{good}
such that $s (d, T, \ep_1, K_1)= \frac{h}{2}$. Now if $K_n$, $n\in
\N$, is constructed, for a more smaller $0< \ep_{n+ 1}< \ep_n$, by
Lemma \ref{good} we let $K_{n+1}$ be a proper compact subset of
$K_n$ with $s (d, T, \ep_{n+1}, K_{n+1})= \frac{h}{n+ 2}$. In fact,
we can require that $\lim_{n\ra +\infty} \ep_n= 0$. Now let
$K_\infty= \{x, y_1, y_2, \cdots\}$ be a subset of $K_0$, where
$y_n\in K_n\setminus K_{n+1}$ for each $n\in\N$.

It is clear that $K_\infty\subseteq K$ is a countable infinite
compact subset, and $$s (d, T, \ep_n, K_\infty)\le s(d, T, \ep_n,
K_n)\le \frac{h}{n}$$ for each $n\in \mathbb{N}$, as
$K_\infty\setminus \{y_1, \cdots, y_{n-1}\}\subseteq K_n$. Hence we
have $$h (T, K_\infty)= \lim_{n\ra +\infty} s(d, T, \ep_n,
K_\infty)= 0.$$ This completes the proof.
\end{proof}
%\begin{rem}
%\end{rem}

\begin{lem}\label{sum}
Let $(X,T)$ be a TDS. Assume that $\{B_n\}_{n\in
\mathbb{N}}\subseteq 2^X$ satisfies $\lim_{n\rightarrow +\infty}
B_n=\{x_0\}$ (in the sense of Hausdorff metric) for some $x_0\in X$
and
\begin{equation} \label{diameter}
\inf_{J\in \mathbb{Z}_+} \lim_{n\rightarrow +\infty} \sup_{j\ge J}
 \text{diam} (T^j B_n)= 0.
\end{equation}
Let $x_n\in B_n$ for each $n\in \N$. Then
\begin{equation} \label{equality}
h\left(T, \bigcup_{n=1}^{+ \infty} B_n \cup \{x_0\}\right)=
\max\left\{\sup_{n\in \mathbb{N}} h(T, B_n), h(T,
\{x_i\}_0^\infty)\right\}.
\end{equation}
If in addition $x_0\not \in B_n$ for each $n\in \mathbb{N}$, then
any countable compact subset of $\bigcup_{n=1}^{+ \infty} B_n \cup
\{x_0\}$ with a unique limit point $x_0$ has entropy at most
$h(T,\{x_i\}_0^\infty)$.
 \end{lem}
\begin{proof}
Let $d$ be a metric on $(X, T)$ and $\epsilon> 0$. Thus by
\eqref{diameter} there exist $i_\epsilon, j_\epsilon\in \mathbb{N}$
such that if $i\ge i_\epsilon$ then $\epsilon_i<
\frac{\epsilon}{4}$, where $\epsilon_i= \sup_{j\ge j_\epsilon}
\text{diam} (T^j B_i)$. Set $X_1= T^{j_\epsilon} (\bigcup_{i\in
\mathbb{N}} B_i \cup\{x_0\})$.

For $n\in\N$, let  $E_n$ be an $(n,\epsilon)$-separated subset of
$X_1$ w.r.t. $T$ with $s_n (d, T, \epsilon, X_1)=|E_n|$. It is clear
that
 if $i\ge i_{\epsilon}$ then $|E_n\cap T^{j_\epsilon} B_{i}|\le 1$,
and if $y\in E_n\cap T^{j_\epsilon} B_i$ then $d_n (y,
T^{j_\epsilon} x_i)\le \epsilon_i< \frac{\epsilon}{4}$. Say $E_n\cap
T^{j_\epsilon} (\bigcup_{i\ge i_{\epsilon}} B_i)= \{ y_1, \cdots,
y_l\}$. For each $1\le r \le l$, there exists $i_r\ge i_{\epsilon}$
such that $y_r\in E_n\cap T^{j_\epsilon} B_{i_r}$. Let $F_n=
\{T^{j_\epsilon}x_{i_1}, \cdots, T^{j_\epsilon} x_{i_r}\}$. For each
$1\le r_1<r_2\le l$,
\begin{equation*}
d_n (T^{j_\epsilon} x_{i_{r_1}}, T^{j_\epsilon} x_{i_{r_2}})\ge d_n
(y_{r_1}, y_{r_2})- (d_n (y_{r_1}, T^{j_\epsilon} x_{i_{r_1}})+ d_n
(y_{r_2}, T^{j_\epsilon} x_{i_{r_2}}))> \frac{\epsilon}{2}.
\end{equation*}
Hence, $F_n$ is an $(n,\frac{\epsilon}{2})$-separated subset of
$T^{j_\epsilon} (\{x_i\}_0^\infty)$ w.r.t. $T$, which implies $s_n
(d, T, \frac{\epsilon}{2}, T^{j_\epsilon} (\{x_i\}_0^\infty))$ $\ge
l$. Note that $G_n\doteq E_n\cap T^{j_\epsilon}
(\bigcup_{i=1}^{i_\epsilon-1} B_i)$ is an $(n,\epsilon)$-separated
subset of $T^{j_\epsilon} (\bigcup_{i=1}^{i_\epsilon-1}B_i)$ w.r.t.
$T$, we have
\begin{eqnarray*}
 s_n (d, T, \epsilon, X_1)&= & |E_n| \le |G_n\cup \{y_1, \cdots,
y_l\}\cup \{T^{j_\epsilon} x_0\}|\le |G_n|+ l+ 1\\
&\le & s_n \left(d, T, \epsilon, T^{j_\epsilon}
\left(\bigcup_{i=1}^{i_\epsilon- 1} B_i\right)\right)+ s_n \left(d,
T, \frac{\epsilon}{2}, T^{j_\epsilon} (\{x_i\}_0^\infty)\right)+ 1,
\end{eqnarray*}
which implies that
\begin{eqnarray*}
& & s \left(d, T, \epsilon, \bigcup_{i\in \mathbb{N}} B_i
\cup\{x_0\}\right)= s
(d, T, \epsilon, X_1) \\
&\le & \max \left\{s \left(d, T, \epsilon, T^{j_\epsilon}
\left(\bigcup_{j= 1}^{i_\epsilon- 1} B_j\right)\right), s \left(d,
T, \frac{\epsilon}{2}, T^{j_\epsilon}
(\{x_i\}_0^\infty)\right)\right\}\\
&= & \max \left\{s \left(d, T, \epsilon, \bigcup_{j= 1}^{i_\epsilon-
1} B_j\right), s \left(d, T,
\frac{\epsilon}{2}, \{x_i\}_0^\infty\right)\right\} \\
&\le & \max \left\{\max_{1\le j\le i_\epsilon- 1} h(T, B_i), h(T,
\{x_i\}_0^\infty)\right\}.
\end{eqnarray*}
Thus we obtain the direction "$\le$" of \eqref{equality}. The other
direction is obvious.

If in additional $x_0\not\in B_n$ for each $n\in \mathbb{N}$, let
$K\subseteq \bigcup_{n=1}^{+ \infty} B_n \cup\{x_0\}$ be any
countable compact subset with a unique limit point $x_0$. Set
$B_n'=K\cap B_n$. Then $B_n'$ is finite for each $n\in \mathbb{N}$,
as $x_0$ is the unique limit point of $K$ and $x_0\not \in B_n'$. So
we have $\{ B_n'\cup\{x_n\}\}_{n\in \mathbb{N}} \subseteq 2^X$ and
\begin{eqnarray*}
h (T, K)& \le & h (T, K\cup \{x_i\}_{i=0}^\infty)=h\left(T,
\bigcup_{n=1}^{+ \infty}
 (B_n' \cup \{x_n\}) \cup \{x_0\} \right) \\
&= & \max\left\{\sup_{n\in \mathbb{N}} h(T, B_n'\cup\{x_n\}), h
(T, \{x_i\}_0^\infty)\right\}\
(\text{using \eqref{equality}}) \\
&= & h(T, \{x_i\}_0^\infty)\ (\text{as $B_n'\cup \{x_n\}$ is a
finite subset for each $n\in \mathbb{N}$}).
\end{eqnarray*}
This finishes the proof.
\end{proof}

\begin{rem}
Without the assumption of \eqref{diameter}, in general Lemma
\ref{sum} doesn't hold.\end{rem}

For example, let $\{x_n\}_{n\in \mathbb{N}}$ be a sequence of $X$
with limit $x\in X$ and $h (T, \{x, x_1, x_2, \cdots\})$ $= a> 0$.

By Theorem \ref{zero} there is a sub-sequence $\{n_i\}_{i\in
\mathbb{N}}$ such that $h(T, \{x, x_{n_1}, x_{n_2}, \cdots\})$ $=
0$. Let $B_j= \{x_{n_{j-1}}, x_{n_{j-1}+ 1}, \cdots, x_{n_{j}- 1}\}$
for each $j\in \mathbb{N}$, where $n_0= 1$. Then $\lim_{j\rightarrow
+\infty} \text{diam} (B_j)= 0$ and $h (T, \bigcup_{j\in \mathbb{N}}
B_j\cup \{x\})= a> 0$, but $\sup_{j\in \mathbb{N}} h (T, B_j) + h(T,
\{x, x_{n_1}, x_{n_2}, \cdots\})= 0$.

In fact, for a good choice of the sub-sequence $\{n_i\}_{i\in
\mathbb{N}}$ in the above construction, we can require that $h (T,
\bigcup_{i\in \mathbb{N}} B_{k_i}\cup \{x\})= a> 0$ for any
sub-sequence $\{k_i\}_{i\in \mathbb{N}}\subseteq \mathbb{N}$. This
is done as follows.

For each $j\in \mathbb{N}$ we can select a sub-sequence
$\{m^{j}_k\}_{k\in \mathbb{N}}\subseteq \N$ such that
\begin{eqnarray*}
& & \lim_{k\rightarrow +\infty} \frac{\log s_{m^{j}_k} (d, T,
\frac{1}{j}, \{x_l: l\ge n_i\})}{m^{j}_k} \\
&= & s \left(d, T, \frac{1}{j}, \{x_l: l\ge n_i\}\right) \left(= s
\left(d, T, \frac{1}{j}, \{x,x_1,x_2,\ldots\}\right)\right)
\end{eqnarray*}
 for each $i\in
\mathbb{N}$. We may assume (replace the sequences $\{n_i\}_{i\in
\mathbb{N}}$ and $\{m^j_k\}_{k\in \mathbb{N}}$ by sub-sequences if
necessary)
\begin{eqnarray*}
s_{m^{j}_i} \left(d, T, \frac{1}{j}, \{x_l: n_i\le l< n_{i+
1}\}\right)\ge e^{m^{j}_i (s (d, T, \frac{1}{j}, \{x, x_1, x_2,
\cdots\})- \frac{1}{i})}\ \text{if}\ 1\le j\le i.
\end{eqnarray*}
Let $B_j= \{x_{n_{j-1}}, x_{n_{j-1}+1}, \cdots, x_{n_{j}- 1}\}$.
Then $$\sup_{j\in \mathbb{N}} h (T, B_j) + h(T, \{x, x_{n_1},
x_{n_2}, \cdots\})= 0.$$ Now for any sub-sequence $\{k_i\}_{i\in
\mathbb{N}}\subseteq \mathbb{N}$ we have: if $l\in \mathbb{N}$ and
$1\le j\le k_l$ then
\begin{eqnarray*}
s_{m^{j}_{k_l}} \left(d, T, \frac{1}{j}, \bigcup_{i\in \mathbb{N}}
B_{k_i}\cup \{x\}\right)\ge s_{m^{j}_{k_l}} \left(d, T, \frac{1}{j},
B_{k_l}\right)\ge e^{m^{j}_{k_l} (s (d, T, \frac{1}{j}, \{x, x_1,
x_2, \cdots\})- \frac{1}{k_l})},
\end{eqnarray*}
which implies that for each fixed $j\in \mathbb{N}$
\begin{eqnarray*}
& &s \left(d, T, \frac{1}{j}, \{x, x_1, x_2, \cdots\}\right)\ge s
\left(d, T, \frac{1}{j},
\bigcup_{i\in \mathbb{N}} B_{k_i}\cup \{x\}\right) \\
&\ge & \limsup_{l\rightarrow +\infty} \frac{1}{m^j_{k_l}} \log
s_{m^{j}_{k_l}} \left(d, T, \frac{1}{j}, \bigcup_{i\in \mathbb{N}}
B_{k_i}\cup \{x\}\right) \ge s \left(d, T, \frac{1}{j}, \{x, x_1,
x_2, \cdots\}\right).
\end{eqnarray*}
Then letting $j\rightarrow +\infty$ we have $h (T, \bigcup_{i\in
\mathbb{N}} B_{k_i}\cup \{x\})= h (T, \{x, x_1, x_2, \cdots\})= a$.

\medskip
Now we are ready to prove the main result in this section.

\begin{thm} \label{haech}
Each \HUL\ TDS is asymptotically h-expansive.
\end{thm}
\begin{proof}
Let $(X, T)$ be a \HUL\ TDS with metric $d$. Assume the contrary
that $(X, T)$ is not asymptotically $h$-expansive, i.e. $a= h^*(T)=
\lim_{\ep\ra 0+} h_T^*(\epsilon)> 0$. Then there exist a sequence
$\{x_i\}_{i\in \mathbb{N}}\subseteq X$ with limit $x$ and a sequence
$\{\epsilon_i\}_{i\in \mathbb{N}}$ of positive numbers with limit
$0$ such that $\lim_{i\rightarrow +\infty} h (T, \Phi_{\epsilon_i}
(x_i))= a$. There are two cases.

\medskip

\noindent{\bf Case 1.} There exist infinitely many $i\in \mathbb{N}$
such that for which $x\not \in \Phi_{\epsilon_i} (x_i)$. Thus w.l.g.
we assume $x\not \in \Phi_{\epsilon_i}(x_i)$ for each $i\in
\mathbb{N}$.

Since $(X, T)$ is \HUL, for each $i\in \mathbb{N}$ we can take a
countable infinite compact subset $X_i\subseteq \Phi_{\epsilon_i}
(x_i)$ with a unique limit point $y_i$ such that $a_i\doteq h (T,
X_i)< a$ and $\lim_{i\rightarrow +\infty} a_i= a$. Clearly,
$\lim_{i\rightarrow +\infty} y_i= x$. Moreover, by Theorem
\ref{zero} we may assume that $h (T, \{x, y_1, y_2, \cdots\})= 0$
(if necessary we take a sub-sequence).

By Lemma \ref{sum}, $h (T, \bigcup_{i\in \mathbb{N}} X_i\cup
\{x\})=\max\left\{\sup_{i\in \mathbb{N}} h(T, X_i), h(T,
\{x,y_1,y_2,\cdots\})\right\}$ $= a$, so there exists a countable
infinite compact subset $A\subseteq \bigcup_{i\in \mathbb{N}}
X_i\cup \{x\}$ with a unique limit point $z$ such that $h (T, A)= a>
0$. By assumptions, if $z= x$ then each $A_i\doteq A\cap X_i$ is a
finite subset, which implies
\begin{eqnarray*}
& &h (T, A)\le  h (T, A\cup \{y_1, y_2, \cdots\}) \\
&= & \max \left\{\sup_{i\in \mathbb{N}} h (T, A_i\cup \{y_i\}), h
(T, \{x, y_1, y_2, \cdots\})\right\}\ (\text{using Lemma
\ref{sum}})= 0,
\end{eqnarray*}
a contradiction with $h (T, A)= a> 0$. Thus $z\not =x$, and so $z\in
X_{v}$ for some $v\in \mathbb{N}$. Put $r= \frac{d (z, x)}{2}$, then
$r> 0$ follows from $x\notin \Phi_{\epsilon_v} (x_v)$. Since $\{x\}$
is the limit of $\{X_i\}_{i\in \mathbb{N}}$ (in the sense of
Hausdorff metric $H_d$), there exists $L\in \mathbb{N}$ such that if
$i> L$ then $d (z, X_i)\ge d (z, x)- H_d (X_i, \{x\})> r$. Thus
$(\bigcup_{i> L} X_i\cup \{x\}) \cap A$ is a finite set, as $z$ is
the unique limit point of $A$ and $d(z,(\bigcup_{i>L} X_i\cup \{x
\}) \cap A)\ge r$. Therefore $$h (T, A)= h \left(T, A\cap
\bigcup_{1\le i\le L} X_i\right)\le \max_{1\le i\le L} h (T, X_i)=
\max_{1\le i\le L} a_i< a,$$ a contradiction.

\medskip

\noindent{\bf Case 2.} $x\in \Phi_{\epsilon_i}(x_i)$ for each large
enough $i\in \mathbb{N}$. W.l.g. we assume that for each $i\in
\mathbb{N}$, $x\in \Phi_{\epsilon_i} (x_i)$ and so
$\Phi_{\epsilon_i} (x_i)\subseteq \Phi_{2\epsilon_i} (x)$. So
$\lim_{\epsilon\rightarrow 0 +} h (T, \Phi_{\epsilon} (x))= a$. As
$h (T, \Phi_{\epsilon} (T^k x))\ge h (T, \Phi_{\epsilon} (x))$ for
each $k\in \mathbb{N}$ and $\epsilon> 0$, then
$\lim_{\epsilon\rightarrow 0 +} h (T, \Phi_{\epsilon} (T^k x))= a$
for each $k\in \mathbb{N}$.

If $\{x, T x, T^2 x, \cdots\}$ is infinite, we fix a point $y\in
\omega (x, T)\doteq \bigcap_{n\in \mathbb{N}} \overline{\{T^j x:
j\ge n\}}$. Then for each $i\in \mathbb{N}$ there exists $k_i\in
\mathbb{N}$ such that $0< d (y, T^{k_i} x)< \frac{1}{i}$. For each
$i\in \mathbb{N}$, as $\lim_{\epsilon\rightarrow 0+} h (T,
\Phi_{\epsilon} (T^{k_i} x))= a$, we may take $0< \eta_i< d (y,
T^{k_i} x)$ such that $h (T, \Phi_{\eta_i} (T^{k_i} x))> \min \{a
(1- \frac{1}{i}), i\}$. Let $y_i= T^{k_i} x$, then
$\lim_{i\rightarrow +\infty} y_i= y$ and $\lim_{i\rightarrow
+\infty} h (T, \phi_{\eta_i} (y_i))= a$ and $y\notin \Phi_{\eta_i}
(y_i)$ for each $i\in \mathbb{N}$. By a similar proof to Case 1, it
is impossible. Hence, $x$ must have a finite orbit, we may assume
that $x$ is a periodic point (if necessary we replace $x$ by $T^k x$
for some $k\in \mathbb{N}$).

Let $l\in \mathbb{N}$ be the period of $x$. Since $T
(\Phi_{\epsilon} (T^k x))\subseteq \Phi_{\epsilon} (T^{k+ 1} x)$ for
each $k\in \mathbb{N}$ and $\epsilon>0$, $\bigcup_{i= 0}^{l- 1}
\Phi_\epsilon (T^i x)$ is compact and $T$-invariant (i.e. $T
(\bigcup_{i= 0}^{l- 1} \Phi_\epsilon (T^i x))\subseteq \bigcup_{i=
0}^{l- 1} \Phi_\epsilon (T^i x)$) for any $\epsilon> 0$. For each
$n\in \mathbb{N}$, let $Y_n= \bigcup_{i=0}^{l- 1} \Phi_{\frac{1}{n}}
(T^i x)$. Then $(Y_n, T)$ is a sub-system of $(X,T)$ and
$h_{\text{top}}(T, Y_n)\ge h (T, \Phi_{\frac{1}{n}}(x))\ge a$. By
the variational principle, there exists $\mu_n\in \mathcal{M}^e
(Y_n, T)$ such that $h_{\mu_n} (T, Y_n)>\min
\{a(1-\frac{1}{n}),n\}$. Obviously, $\mu_n(\{x\})=0$ and
$\mu_n(\Phi_{\frac{1}{n}}(x))\ge \frac{1}{l}$. Thus, we can take a
compact subset $K_n\subseteq \Phi_{\frac{1}{n}}(x)$ such that
$\mu_n(K_n)>0$ and $x\notin K_n$. By a classic result of Katok
\cite[Theorem 1.1]{Ka} (see also \cite[Theorem 3.7]{YZ}), we know
that
$$h(T, K_n)\ge h_{\mu_n}(T, Y_n)>\min \left\{a\left(1-\frac{1}{n}\right), n\right\},$$ as
$\mu_n (K_n)>0.$ As $(X, T)$ is \HUL, there is a countable compact
subset $X_n\subseteq K_n$ with a unique limit point $y_n$ such that
$h (T, X_n)= a_n= \min \{a(1-\frac{1}{n}), n\}< a$. Clearly,
$x\notin X_n$ for each $n\in \mathbb{N}$, as $x\notin K_n$. Again by
a similar proof to Case 1, we know that this is impossible. Thus,
$(X, T)$ must be asymptotically $h$-expansive.
\end{proof}

\section{Properties preserved by a principal extension}

Combined with the results obtained in sections 5 and 6, in this
section we shall answer question \ref{q 3} by proving that a TDS is
\HUL\ iff it is asymptotically $h$-expansive. Moreover, we present a
hereditarily lowerable TDS with finite entropy which is not \HUL. As
a byproduct, we show that principal extension preserves the
properties of lowering, hereditary lowering and \HUL.

\medskip

Let $(X, T)$ and $(Y, S)$ be GTDSs. We say that $\pi: (X,
T)\rightarrow (Y, S)$ is a {\it factor map} if $\pi$ is a continuous
surjective map and $\pi\circ T= S\circ \pi$. Let $\pi: (X,
T)\rightarrow (Y, S)$ be a factor map between GTDSs.  The {\it
relative topological entropy of $(X, T)$ w.r.t. $\pi$} is defined as
follows:
\begin{equation*}
h_{\text{top}} (T, X| \pi)= \sup_{y\in Y} h(T,\pi^{-1}(y)).
\end{equation*}

Let $\pi:(X,T)\rightarrow (Y,S)$ be a factor map between GTDSs.
It's easy to check that on $Y$ the function $y\mapsto h (T,
\pi^{-1} (y))$ is $S$-invariant and Borel measurable. Thus for
each $\nu\in \mathcal{M} (Y,S)$ we may define
\begin{equation*}
h (T, X| \nu)= \int_{Y} h (T, \pi^{-1} (y)) d \nu (y).
\end{equation*}
In particular, if $\nu$ is ergodic then $h (T, \pi^{-1} (y))= h (T,
X| \nu)$ for $\nu$-a.e $y\in Y$. Thus

\begin{prop}
Let $\pi: (X, T)\rightarrow (Y, S)$ be a factor map between GTDSs.
Then for each $\nu\in \mathcal{M}^e (Y,S)$ there exists a
countable compact set $K\in 2^X$ with $h (T, K)= h (T, X| \nu)$.
\end{prop}

Now let $\pi:(X,T)\rightarrow (Y,S)$ be a factor map between GTDSs
and $\mu\in \mathcal{M} (X, T)$, note that the sub-$\sigma$-algebra
$\pi^{-1} (\mathcal{B}_Y)\subseteq \mathcal{B}_X$ satisfies
$T^{-1}\pi^{-1} (\mathcal{B}_Y)\subseteq \pi^{-1} (\mathcal{B}_Y)$
in the sense of $\mu$, we define {\it relative measure-theoretical
$\mu$-entropy of $(X, T)$ w.r.t. $\pi$} as
\begin{equation*}
 h_\mu (T, X| \pi)= h_\mu (T, X| \pi^{- 1} (\mathcal{B}_Y)).
\end{equation*}

The following results are proved in \cite{DS} and \cite{LW1}.

\begin{lem} \label{BD22}
Let $\pi: (X,T)\rightarrow (Y,S)$ be a factor map between GTDSs.
Then
\begin{enumerate}

\item One has
\begin{equation*}
\hskip 26pt h_{\text{top}} (T,X|\pi)= \sup_{\nu\in
\mathcal{M}(Y,S)} h (T, X| \nu)= \sup_{\nu\in \mathcal{M}^e (Y,S)}
h (T, X| \nu).
\end{equation*}

\item For each $\nu\in \mathcal{M} (Y,S)$, $$h (T, X| \nu)=
\sup \{h_\mu (T, X| \pi): \mu\in \mathcal{M} (X, T), \pi\mu=
\nu\}.$$

\item For each $\mu\in \mathcal{M}(X,T)$, $h_\mu (T, X)=
h_\mu (T, X| \pi)+ h_{\pi\mu} (S, Y)$.
\end{enumerate}
\end{lem}

We have proved that each expansive TDS is \HUL. In fact, the same
conclusion holds for a more general case. To prove this, first we
shall prove the following Bowen's type theorem which is interesting
itself. We remark that the idea of the proof is inspired by the
proof of \cite[Theorem 17]{B}.

\begin{thm} \label{Fe}
Let $\pi:(X,T)\rightarrow (Y,S)$ be a factor map between GTDSs and
$E\in 2^X$. Then
\begin{equation*}
h (S, \pi (E))\le h (T, E)\le h (S, \pi (E))+ h_{\text{top}} (T, X|
\pi).
\end{equation*}
In particular, if $(Y,S)$ has finite entropy then for each $E\in
2^X$, $$h(T,E)-h(S,\pi(E))\le h_{\text{top}}(T,X|\pi).$$
\end{thm}
\begin{proof}
By the continuity of $\pi$, it's easy to obtain $h(S,\pi(E))\le
h(T,E)$. Thus it remains to prove $h(T,E)\le h(S,\pi(E))+
h_{\text{top}} (T,X|\pi).$ If $h_{\text{top}} (T,X|\pi)=+\infty$,
this is obvious. Now we suppose that $a= h_{\text{top}} (T,X|\pi)<
+\infty$. Let $d_X$ and $d_Y$ be the metrics on $(X, T)$ and $(Y,
S)$, respectively.

Let $\epsilon>0$ and $\alpha>0$. By Lemma \ref{BD22} (1), for each
$y\in Y$ we may choose $m(y)\in \mathbb{N}$ such that
\begin{equation} \label{selection}
a+ \alpha\ge h (T, \pi^{-1} (y))+ \alpha\ge \frac{1}{m(y)} \log
r_{m(y)} (d_X, T, \epsilon, \pi^{-1} (y)).
\end{equation}
Let $E_y$ be an $(m(y),\epsilon)$-spanning set of $\pi^{-1}(y)$
w.r.t. $T$ with the minimum cardinality. Then $U_y\doteq
\bigcup_{z\in E_y} B_{m(y)}(z,2\epsilon)$ is an open neighborhood of
$\pi^{-1}(y)$, where $B_{m(y)}(z,2\epsilon)$ denotes the open ball
in $X$ with center $z$ and radius $2\epsilon$ (in the sense of
$(d_X)_{m (y)}$-metric). Since the map $\pi^{-1}: Y\rightarrow 2^X$,
$y\mapsto \pi^{-1}(y)$ is upper semi-continuous, there exists an
open neighborhood $W_y$ of $y$ for which $\pi^{-1}(W_y)\subseteq
U_y$. By the compactness of $Y$ there exist $\{y_1, \cdots,
y_k\}\subseteq Y$ such that $\mathcal{W}\doteq
\{W_{y_1},\cdots,W_{y_k}\}$ forms an open cover of $Y$. Let $\delta>
0$ be a Lebesgue number of $\mathcal{W}$ and $M=\max
\{m(y_1),\cdots,m(y_k) \}$.

Let $\pi(E)_n$ be any $(n,\delta)$-spanning set for $\pi(E)$ w.r.t.
$S$ with the minimum cardinality. For each $y\in \pi(E)_n$ and $0\le
j<n$, pick $c_j(y)\in \{y_1,\cdots,y_k\}$ with $\overline{B
(S^{j}(y), \delta)}\subseteq W_{c_j(y)}$, where $B (S^{j}(y),
\delta)$ denotes the open ball in $Y$ with center $S^{j}(y)$ and
radius $\delta$. Now define recursively $t_0(y)=0$ and
$t_{s+1}(y)=t_s(y)+m(c_{t_s(y)}(y))$ ($s\in \mathbb{Z}_+$) until one
gets a $t_{q+1}(y)\ge n$; set $q(y)=q\le t_q (y)$.

For any $y\in \pi(E)_n$ and $z_0\in E_{c_{t_0(y)}(y)}, z_1\in
E_{c_{t_1(y)}(y)},\cdots, z_{q (y)}\in E_{c_{t_{q(y)} (y)}(y)}$ we
define
\begin{equation*}
V(y;z_0,z_1,\cdots,z_{q(y)})=\{u\in X: (d_X)_{m(c_{t_s(y)}(y))}
(T^{t_s(y)}(u),z_s)<2\epsilon, 0\le s\le q(y)\}.
\end{equation*}
Obviously, for each $y\in \pi(E)_n$, the number of permissible
tuples
 $(z_0,z_1,\cdots,z_{q(y)})$ is
\begin{equation} \label{estimate 4}
N_{y}=\prod_{s=0}^{q(y)}
 r_{m(c_{t_s(y)}(y))}(d_X, T,\epsilon,\pi^{-1}(c_{t_s(y)}(y))).
\end{equation}
Then we have (using \eqref{selection} and \eqref{estimate 4})
\begin{equation} \label{estimate 5}
N_y\le \prod_{s=0}^{q(y)} e^{(a+ \alpha) m (c_{t_s(y)}(y))}= e^{(a+
\alpha) (t_{q (y)} (y)+ m (c_{t_{q (y)} (y)}(y)))}\le
e^{(a+\alpha)(n+M)}.
\end{equation}
Note that if $F$ is an $(n,4\epsilon)$-separated subset of $E$
w.r.t. $T$ then, for each permissible tuple $(z_0, z_1, \cdots,
z_{q(y)})$, $V(y; z_0,z_1,\cdots,z_{q(y)})\cap F$ has at most one
element, and
\begin{equation*}
\bigcup_{y\in \pi(E)_n} \left(\bigcup_{z_s\in E_{c_{t_i (y)} (y)},
0\le s\le q (y)} V(y; z_0,z_1,\cdots,z_{q(y)})\right)\supseteq E.
\end{equation*}
Thus combining \eqref{estimate 5} we have
\begin{equation*}
s_n (d_X, T, 4\epsilon, E)\le \sum_{y\in \pi(E)_n} N_y\le r_n
(d_Y, S, \delta, \pi (E)) e^{(a+\alpha)(n+M)}.
\end{equation*}
Letting $n\rightarrow +\infty$ one has $s (d_X, T,4\epsilon,E)\le
r(d_Y, S,\delta,\pi(E))+a+\alpha\le h(S,\pi(E))+a+\alpha$. Since
$\epsilon>0$ and $\alpha>0$ are arbitrary, we obtain $h (T,E)\le
h(S,\pi(E))+a$. This finishes the proof.
\end{proof}

As a direct consequence of Theorem \ref{Fe}, we have the following
proposition.

\begin{prop} \label{Fe11}
Let $\pi:(X,T)\rightarrow (Y,S)$ be a factor map between GTDSs. If
$(Y,S)$ has finite entropy, then
\begin{equation*}
\sup_{E\in 2^X} (h(T,E)-h(S,\pi(E))=h_{\text{top}}(T,X|\pi).
\end{equation*}
\end{prop}
\begin{proof} On one hand  we know
$\sup_{E\in 2^X} (h(T,E)-h(S,\pi(E))\le h_{\text{top}}(T,X|\pi)$ by
Theorem \ref{Fe}. On the other hand, we have
\begin{align*}
\sup_{E\in 2^X} (h(T,E)-h(S,\pi(E))&\ge \sup_{y\in Y}
(h(T,\pi^{-1}(y))-h(S,\{y\}))\\
&=\sup_{y\in Y} h(T,\pi^{-1}(y))= h_{\text{top}}(T,X|\pi).
\end{align*} This completes the proof.
\end{proof}

%Following the same idea of the proof of Theorem \ref{Fe} we can
%prove the following Theorem.
%
%\begin{thm} \label{observation 06.01.08}
%Let $(X, T)$ be a TDS. Assume that $(X_1, T)$ is a sub-system of
%$(X, T)$ admitting a family of non-empty compact subsets
%$\{K_x\}_{x\in X_1}$ of $X$ such that $\bigcup_{x\in X_1} K_x= X$
%and the map $X_1\rightarrow 2^X, x\mapsto K_x$ is upper-semi
%continuous. If for each $x\in X_1$ given a point $k_x\in K_x$ then
%one has that for each $X_2\subseteq X_1$
%\begin{equation*}
%h (T, \{k_x: x\in X_2\})\le h \left(T, \bigcup_{x\in X_2}
%K_x\right)\le h (T, \{k_x: x\in X_2\})+ \sup_{x\in X_1} h (T, K_x).
%\end{equation*}
%In particular, if $(X_1, T)$ has finite entropy then
%$$\sup_{X_2\in 2^{X_1}} \left(h
%\left(T, \bigcup_{x\in X_2} K_x\right)- h (T, \{k_x: x\in
%X_2\})\right)= \sup_{x\in X_1} h (T, K_x).$$
%\end{thm}

Let $\pi: (X, T)\rightarrow (Y, S)$ be a factor map between GTDSs.
We call that $\pi$ is a {\it principal factor map} (or $(X, T)$ is
a {\it principal extension of $(Y, S)$}) if $h_{\mu} (T, X)=
h_{\pi \mu} (S, Y)$ for each $\mu\in \mathcal{M} (X, T)$. This was
introduced and studied firstly by Ledrappier \cite{Le}.

The following result is a direct consequence of Lemma \ref{BD22}
and Theorem \ref{Fe}.

\begin{cor} \label{corollary-06.05.23}
Let $\pi: (X, T)\rightarrow (Y, S)$ be a principal factor map
betweens GTDSs. If $(Y, S)$ has finite entropy then
$h_{\text{top}}(T, X| \pi)= 0$ and $h (T, K)= h (S, \pi (K))$ for
all $K\in 2^{X}$.
\end{cor}

A characterization of asymptotical $h$-expansiveness is obtained
recently by Boyle and Downarowicz as the following \cite[Theorem
8.6]{BD}:

\begin{lem} \label{BD}
A TDS $(X,T)$ is  asymptotically $h$-expansive iff it admits a
principal extension to a symbolic TDS.
\end{lem}

Then question \ref{q 3} is answered as follows:

\begin{thm} \label{ex-general}
A TDS $(X, T)$ is asymptotically $h$-expansive iff it is \HUL.
\end{thm}
\begin{proof}
First, each \HUL\ TDS is asymptotically $h$-expansive by Theorem
\ref{haech}.

Now assume that TDS $(X, T)$ is asymptotically $h$-expansive and by
Lemma \ref{BD} let $\pi: (X', T')\rightarrow (X, T)$ be a principal
factor map with $(X', T')$ a symbolic TDS. Then $h_{\text{top}} (T,
X)\le h_{\text{top}} (T', X')< +\infty$ (as $(X', T')$ is a symbolic
TDS), and so for each $E\in 2^{X'}$, we have  $h (T', E)= h (T, \pi
(E))$ by Corollary \ref{corollary-06.05.23}. Given $E\in 2^X$. Since
$(X', T')$ is an expansive TDS, then using Theorem \ref{expansive}
we have that for each $0\le h\le h (T, E)= h (T', \pi^{- 1} (E))$
there exists a countable compact subset $X'_h\subseteq \pi^{- 1}
(E)$ with at most a limit point in $X'$ such that $h (T', X'_h)= h$.
Now set $X_h= \pi (X'_h)\subseteq E$. So $X_h\subseteq X$ is a
countable compact subset with at most a limit point in $X$ and $h
(T, X_h)= h (T', X'_h)= h$. That is, TDS $(X, T)$ is \HUL.
\end{proof}

Moreover, combining with Corollary \ref{corollary-06.05.23} and
Theorem \ref{ex-general} we claim that principal extension preserves
the properties of lowering, hereditary lowering and \HUL.

\begin{prop} \label{standard}
Let $\pi: (X', T')\rightarrow (X, T)$ be a principal factor map
between TDSs. If $(X, T)$ has finite entropy then
\begin{enumerate}

\item $(X', T')$ is asymptotically $h$-expansive iff so is $(X, T)$.

\item $(X', T')$ is lowerable (resp. hereditarily lowerable,
\HUL) iff so is $(X, T)$.
\end{enumerate}
\end{prop}
\begin{proof}
(1) is only a special case of Ledrappier's result about principal
extensions \cite[Thorem 3]{Le}. (2) follows directly from Corollary
\ref{corollary-06.05.23}, Theorem \ref{ex-general} and (1).
\end{proof}

%\section{Remarks on hereditary lowering}

It is not hard to construct examples with infinite entropy which
are hereditarily lowerable. Thus, there are TDSs which are
hereditarily lowerable but not \HUL. In fact, an example with the
same property which has finite entropy exists.

\begin{ex} \label{example 1}
There exists a hereditarily lowerable TDS $(X, T)$ with finite
entropy which is not \HUL. The detailed construction is given as
follows:
\end{ex}

Take a countable copies of the full shift over $\{0, 1\}^\Z$ and
embed them into $B_n$ with $\{B_n\}_{n\in \mathbb{N}}$ a sequence of
disjoint compact balls in $\mathbb{R}^2$ such that $(0, 0)\notin
B_n\ra \{(0, 0)\}$ (in the sense of Hausdorff metric). Let $(X, T)$
be the TDS of the union of $\{(0, 0)\}$ with these copies, where $T$
is the shift if it is restricted on each copy and $(0, 0)\mapsto (0,
0)$. Then $h_{\text{top}} (T, X)= \log 2$ (using the classical
variational principle). For each copy we may take $C_n\subseteq B_n$
with entropy $a_n< \log 2$ such that $\lim_{n\rightarrow +\infty}
a_n= \log 2$ (using Theorem \ref{aim of paper}). Then $h (T,
\bigcup_{n\in \mathbb{N}} C_n\cup \{(0, 0)\})= \log 2$. Whereas, by
definition it is not hard to see that any countable compact subset
of $X$ with a unique limit point $(0, 0)$ must have zero entropy,
which implies that each countable compact subset of $\bigcup_{n\in
\mathbb{N}} C_n\cup \{(0, 0)\}$ with a unique limit point has
entropy smaller strictly than $\log 2$. Thus $(X, T)$ is not \HUL.

 Now we claim
that $(X, T)$ is hereditarily lowerable. For each $n\in \mathbb{N}$,
we take $x_n\in B_n$. Then it is not hard to see that
$h(T,\{x_n\}_{n\in \mathbb{N}}\cup \{ (0,0)\})=0$. Let $K\in 2^X$
and $0\le h \le h (T, K)$. For each $n\in \mathbb{N}$, set
$K_n=K\cap B_n$. Since $B_n$ is \HUL\ (see Theorem \ref{expansive}),
we may take $K_n^h\in 2^{K_n}$ with $h (T, K_n^h)= \min \{h, h (T,
K_n)\}$. Using Lemma \ref{sum} we have
\begin{eqnarray*}
\sup_{n\in \mathbb{N}} h (T, K_n)&=& \max \left\{\sup_{n\in
\mathbb{N}} h(T, K_n\cup \{x_n\}),h(T,\{x_n\}_{n\in \mathbb{N}}\cup
\{
(0,0)\})\right\}\\
&=& h \left(T, \bigcup_{n\in \mathbb{N}} (K_n\cup \{x_n\})
\cup\{(0,0)\}\right)\ge h(T, K)\ge h.
\end{eqnarray*}
This implies
 $\sup_{n\in
\mathbb{N}} h (T, K_n^h)= h$. Let $K_h= \bigcup_{n\in \mathbb{N}}
K_n^h\cup\{(0,0)\}\in 2^K$. Then using Lemma \ref{sum} again one has
$h (T, K_h)= \sup_{n\in \mathbb{N}} h (T, K_n^h)= h$. This means
that $K$ is lowerable, and so $(X, T)$ is a hereditarily lowerable
TDS. This ends the example.

\medskip
It is not difficult to show that the above example (by a small
modification) has a symbolic extension with the same entropy, which
is not a principal one (see \cite{BD} for other examples of the same
type). Thus it is an interesting question if each system having a
symbolic extension is hereditarily lowerable.
%\newpage

\section{Appendix}
In this Appendix we want to explain that our main results hold for
GTDSs.  Note that we can also introduce the lowerable, hereditarily
lowerable, {\bf HUL} and asymptotically $h$-expansive properties for
a GTDS. Let $(X,T)$ be a GTDS. If $T$ is surjective, we can use the
standard natural extension as follows:

Assume that $d$ is a metric on $X$. We say $(X_T,S)$ is the {\it
natural extension} of $(X,T)$, if $X_T=\{ (x_1, x_2, \cdots):
T(x_{i+1})=x_i, x_i\in X, i\in \N \}$, which is a sub-space of the
product space $\prod_{i=1}^\infty X$ with the compatible metric
$d_T$ defined by
$$d_T((x_1, x_2, \cdots),(y_1,y_2,\cdots))=\sum_{i=1}^\infty
\frac{d(x_i,y_i)}{2^i}.$$ Moreover, $S:X_T\ra X_T$ is the shift
homeomorphism, i.e. $S(x_1, x_2, x_3, \cdots)= (T(x_1), x_1, x_2,
\cdots)$. The following observation is easy.

\begin{thm} Let $(X_T,S)$ be the natural extension of $(X,T)$ with $T$
surjective. Then $(X_T,S)$ is lowerable (resp. hereditarily
lowerable, {\bf HUL}, asymptotically $h$-expansive) iff so is
$(X,T)$.
\end{thm}
\begin{proof} Let $\pi_1: X_T\ra X$ be the projection to the first
coordinate. Observe that
$\text{diam} (S^n\pi_1^{-1}(x))$ $\ra 0$ for each $x\in X$. This
implies that $h(S, \pi_1^{-1}(x))=0$ for each $x\in X$, and hence
$h_{\text{top}}(S, X_T|\pi_1)=\sup_{x\in X}h(S, \pi_1^{-1}(x))=0$.
By Theorem \ref{Fe}, $(X_T,S)$ is lowerable (resp. hereditarily
lowerable, {\bf HUL}) iff so is $(X,T)$.

Since $h(S, \pi_1^{-1}(x))=0$ for each $x\in X$, $\pi_1$ is a
principal extension by Lemma \ref{BD22}. Now as a principal
extension preserves the property of asymptotical $h$-expansiveness
(see \cite[Theorem 3]{Le}), $(X_T,S)$ is asymptotically
$h$-expansive iff so is $(X,T)$.
\end{proof}

If $T$ is not surjective, we will construct a surjective system
$(X', T')$ such that the dynamical properties of $(X,T)$ and $(X',
T')$ are 'very close' as follows:

Let $X'=X\times \{0\}\cup X\times \{\frac{1}{n}: n\in\N\}$.
Moreover, put $T'(x,0)=(x, 0), T'(x, \frac{1}{n+1})=(x,
\frac{1}{n})$ and $T'(x,1)=(Tx,1)$ for $n\in \N$ and $x\in X$.

It is not hard to check that $(X',T')$ is lowerable (resp.
hereditarily lowerable, {\bf HUL}, asymptotically $h$-expansive) iff
so is $(X,T)$. Collecting terms, one has
\begin{thm} Let $(X,T)$ be a GTDS. Then $(X,T)$ is \HUL\ iff it is asymptotically $h$-expansive, and if $(X,T)$ has
finite entropy then it is lowerable.
\end{thm}

\end{document}